\documentclass[11pt]{article}
\usepackage[utf8]{inputenc}

\usepackage[a4paper]{geometry}
\usepackage[dvipsnames]{xcolor}
\usepackage{amsfonts,amsmath,amssymb,amsthm,mathrsfs,bbm,mathtools,dsfont,appendix,float,lineno,setspace} 
\usepackage{subfig}
\usepackage[shortlabels]{enumitem}
\usepackage[pagebackref=true,colorlinks=true, linkcolor=RoyalPurple, citecolor=RoyalPurple]{hyperref}
\usepackage[nameinlink]{cleveref} 
\usepackage{natbib}
\usepackage{stmaryrd}
\usepackage{tikz}

\renewcommand*\backref[1]{
}

\theoremstyle{plain}
\newtheorem{definition}{Definition}
\newtheorem{proposition}[definition]{Proposition}
\newtheorem{lemma}[definition]{Lemma}
\newtheorem{corollary}[definition]{Corollary}
\newtheorem{theorem}{Theorem}

\theoremstyle{definition}

\numberwithin{definition}{section}
\numberwithin{equation}{section} 

\newcommand{\be}{\overline{e}}
\renewcommand*{\P}{\mathbb{P}}
\newcommand*{\E}{\mathbb{E}}
\newcommand*{\R}{\mathbb{R}}

\newcommand*{\Z}{\mathbb{Z}}

\newcommand*{\N}{\mathbb{N}}

\newcommand*{\Bcal}{\mathcal{B}}
\newcommand*{\Ecal}{\mathcal{E}}

\newcommand*{\Fcal}{\mathcal{F}}

\newcommand*{\Gcal}{\mathcal{G}}
\newcommand*{\Ccal}{\mathcal{C}}

\newcommand*{\Scal}{\mathcal{S}}

\newcommand*{\Hcal}{\mathcal{H}}

\newcommand{\norm}[1]{\left \lVert  #1 \right \rVert}

\newcommand*{\1}{\mathds{1}}
\newcommand{\x}{\boldsymbol{x}}
\newcommand*{\y}{\boldsymbol{y}}

\renewcommand*{\d}{\mathrm{d}}

\newcommand{\ssup}[1] {{{\scriptscriptstyle{{#1}}}}} 
\newcommand{\Ssup}[1]{^{\ssup{(#1)}}}

\DeclareMathOperator{\Label}{Label}
\DeclareMathOperator{\Type}{Type}
\DeclareMathOperator{\Int}{Int}

\crefname{equation}{}{}

\title{First-order phase transition for Gibbs point processes with saturated interactions}

\author{D.~Dereudre, C.~Renaud-Chan}

\date{\today}

\begin{document}
	
	\maketitle
	
	\begin{abstract}
		We study first-order phase transitions in continuum Gibbs point processes with saturated interactions. These interactions form a broad class of Hamiltonians in which the local energy in regions of high particle density depends only on the number of points. Building on ideas of Pirogov–Sinaĭ-Zahradnik theory and its adaptations to the continuum, we develop a general method for establishing the existence of two distinct infinite-volume Gibbs measures with different intensities in this setting, demonstrating a first-order phase transition. Our approach extends previous results obtained for the Quermass model and applies in particular to a new class of diluted pairwise interactions introduced in this work. 
	\end{abstract}
	
	\section{Introduction}
	
	A grand canonical Gibbs point process on a finite volume $\Delta \subset \mathbb{R}^d$ is a probability measure with unnormalized density given by the Boltzmann factor $e^{-\beta H}$ with respect to the Poisson point process on $\Delta$ with activity $z>0$. Here, $\beta>0$ is the inverse temperature and $H$ is the Hamiltonian. Infinite-volume Gibbs point processes are defined as solutions to the Dobrushin--Lanford--Ruelle equations, which characterize equilibrium states of the system. The question of uniqueness or non-uniqueness of such probability measures on the configuration space is both difficult and has a long history since \cite{RuelleWR}.
	
	It is known that for pairwise interactions in regimes where both $z$ and $\beta$ are small, the Gibbs measure is unique. This has been established through various approaches: in \cite{ruelle1970superstable} using the Kirkwood--Salsburg equations; in \cite{mayer1941molecular} via cluster expansion techniques to prove analyticity of the pressure; in \cite{dobrushin2006criterion} and \cite{Houdebert_Zass_2022} using Dobrushin's uniqueness criterion; and in \cite{betschLastUniqueness} through a coupling with the random connection model.
	
	A first-order phase transition occurs when the pressure of the system is non-differentiable. This phenomenon is related to the non-uniqueness of Gibbs measures, and one can show the existence of two Gibbs measures with different intensities. In the literature, there are only a few results on phase transitions for continuum Gibbs measures without spins. 
	
	The most well-known results concern the Area-Interaction model, which was introduced by \cite{WidomRowlinson} and the phase transition was established using different methods in \cite{RuelleWR}, \cite{2ChayesKotecky}, and \cite{giacomin1995agreement}. Its phase diagram was completed in \cite{HoudebertDereudre}, showing that non-uniqueness occurs only when $z=\beta$; however, the question remains open in a small neighborhood of the critical point. In one dimension, a phase transition for a Lennard--Jones interaction was proved in \cite{johansson1995separation}. We also mention the classical result of \cite{LebowitzMazelPresutti} for Kac-type interactions, whose proof adapts Pirogov-Sinaï theory to the continuum. More recently, a similar phase transition result was obtained for a modified version of this Kac-type interaction in \cite{he2025liquidvaportransitionmodelcontinuum}.
	
	In this paper, we develop a general method to prove first-order phase transitions for saturated interactions. This work generalizes previous results on the Quermass model \cite{renaudchanQuermass} to a broader class of interactions. Our approach is inspired by the ideas and techniques of Pirogov, Sinaï, and Zahradn\'ik \cite{PirogovSinai, PirogovSinai2, zahradnik1984alternate}, adapted to the continuum. This class of saturated interactions corresponds to Hamiltonians for which the energy in regions of high local particle density depends only on the number of points. We introduce a new class of interaction that we call the diluted pairwise interaction. It is characterized by a pair potential and a parameter $R>0$ that we call the dilution scale. The diluted pairwise interaction is an approximation of the pairwise interaction under a proper renormalisation and as $R \to 0$ we recover the pairwise interaction. Using our framework, we establish a new phase transition result for the diluted pairwise interactions. Moreover, for pair potential that is strongly repulsive at short range, such as the Lennard-Jones potential, we show that for any scale of dilution there is always a way to truncate the pair potential at the origin and have a phase transition for this truncated pair potential. This corollary opens a new path to study phase transition for pairwise interactions for which no phase transition results are known. 
	
	Our paper is organized as follows. In Section \ref{Section2}, we introduce the notation, the classes of saturated and diluted pairwise interactions, and provide the definition of Gibbs point processes and the notion of contours. In Section \ref{Section3}, we present the main results of this paper. In Section \ref{sec: proof_thm1}, we prove the general first-order phase transition result for saturated interactions. Finally, in Section \ref{sec: proof_thm2}, we establish that under certain assumptions, the diluted pairwise interaction satisfies a Peierls condition and thus exhibits a first-order phase transition.
	
	\tableofcontents
	
	\section{Definitions and models}\label{Section2}
	
	\subsection{Setting and notations}
	
	We denote by $\Bcal_b(\R^d)$ the set of bounded Borel sets of $\R^d$ with positive Lebesgue measure. For any sets $A$ and $B$ in $\Bcal_b(\R^d)$, $A \oplus B$ stands for the Minkowski sum of these sets, i.e. $A \oplus B = \{a+b, \forall a \in A, \forall b \in B\}$. Let $M$ ne a Polish space describing the mark or the spin state of particles. We define $\Scal := \R^d \times M $ the state space of a single marked point. For any $(x,m)\in\Scal$, the first coordinate $x$ is for the location of the point and the second coordinate $m$ is the mark of a particle. For any subset $C \subset \Scal$ and any $ u\in \R^d$, $C + u := \{(x+u, m), (x,m) \in C \}$ is the translation of $C$ by a vector $u$. For any set $\Delta \in \mathcal{B}_b(\R^d)$, $\Scal_\Delta$ is the local state space $\Delta \times M$. A configuration of marked points $\omega$ is a locally finite set in $\Scal$; i.e. $N_\Delta(\omega) := \#(\omega \cap E_\Delta)$ is finite for any $\Delta \in \Bcal_b(\R^d)$. We denote by $\Omega$ the set of all marked point configurations and by $\Omega_f$ its restriction to finite configurations. For any $\omega\in\Omega$, its projection in $\Delta\subset \R^d$ is defined by $\omega_\Delta:=\omega \cap \Scal_\Delta$. We equip the state space $\Omega$ with the $\sigma$-algebra generated by the counting functions on $\Scal$, $\omega \mapsto \#(\omega \cap (\Delta \times A))$ for $\Delta \in \Bcal_b(\R^d)$ and $ A \in \Bcal(M)$. We consider a reference measure $\lambda \otimes \P_M$ on $\R^d \times M$ where $\lambda$ is the Lebesgue measure on $\R^d$ and $\P_M$ a probability measure on the mark space, $M$, of the particle.  
	A point process $\mathbb{X}$ is a random variable on the set of configurations, $\Omega$, and it is said to be stationary in space if for any $u \in \R^d$, $\mathbbm{X}+u \overset{d}{=} \mathbbm{X}$. Let $z>0$ be the activity, we denote by $\Pi^z$ the Poisson point process with intensity measure $z \lambda \otimes \P_M$. For any $\Delta \in \Bcal_b(\R^d)$, we denote by $\Pi_\Delta^z$ the restriction of the Poisson point process to the configurations inside $\Delta$.
		
	\subsection{Interaction}
	
	Let us introduce the general assumptions on the interaction between particles of our system. The Hamiltonian, also called the energy functional, is a measurable function from $\Omega_f$ to $\R \cup \{+\infty\}$. For any $\Lambda \in \Bcal_b(\R^d)$ and any configuration $\omega \in \Omega$ we define the local energy in $\Lambda$ of the configuration $\omega$ as
	\begin{equation}
		H_\Lambda(\omega) := \lim\limits_{r \to \infty} \left( H(\omega_{\Lambda \oplus B(0,r)}) - H(\omega_{\Lambda \oplus B(0,r) \setminus \Lambda}) \right)
	\end{equation}
	if the limit exists. Similarly we define the local energy of a point $x \in \R^d$ in a configuration $\omega \in \Ccal$ as
	\begin{equation}
		h(x,\omega) := \lim\limits_{r \to \infty} \left( H(\{x\} \cup \omega_{B(x,r)}) - H(\omega_{B(x,r) }) \right)
	\end{equation}
	if the limit exists. We assume that the Hamiltonian satisfies the following set of assumptions: 
	\begin{enumerate}[($\Hcal$1)]\label{H}
		\item {\bf Stability}. There exists $C_s \geq 0$ such that for all $\omega \in \Omega_f$, \ $H(\omega) \geq  - C_s N(\omega)$.  \label{H:stability}
		\item  {\bf Non-degeneracy}. $H(\emptyset) < +\infty$.  \label{H:nondegenerate}
		\item {\bf Heredity}. If for $\omega \in \Omega_f$  $H(\omega) = \infty$ then for any $x \in \R^d$,  $H(\omega \cup \{x\}) = + \infty$. \label{H:heredity}
		\item {\bf Invariance by translation}. For all $u\in\R^d$ and $\omega \in \Omega_f$, we have $H(\omega) = H(\omega + u )$.  \label{H:invariant}
		\item {\bf Finite range}. There is $R>0$ such that for all $\omega \in \Omega$ and $\Delta \in \Bcal_b(\R^d)$, $H_\Delta(\omega) = H_\Delta (\omega_{\Delta \oplus B(0,R)})$.  \label{H:finiterange}
	\end{enumerate}
	These are classical assumptions on the Hamiltonian that are verified by pairwise interactions with Strauss potential or truncated Lennard-Jones potential and also by multibody interaction such as the Area interaction or the Quermass interaction. In this paper, however we will be interested in a subclass of interaction that satisfy a saturation property. 
	
	\subsection{Coarse graining and saturation} \label{subsec: def_coarse_sat}
	
	Let us consider a paving of $\R^d$ with tiles of length $\delta >0$. For any integer $i \in \Z^d$, we denote the $i$-th tile with  $T_i := \left[ -\frac{\delta}{2}, \frac{\delta}{2}\right)^d + i\delta$. For any subset $\Lambda \subset \Z^d$, we denote by $\hat \Lambda = \bigcup_{i \in \Lambda} T_i$. Before going further, we define two sets of locally homogeneous configurations around a tile. Let $L>0$ and $i \in \Z^d$, we define the homogeneous dense (resp.\ empty) configurations around $T_i$,
	\begin{gather}
		\Omega_{i,L}\Ssup{1} := \{ \omega \in \Omega, \omega_{T_i} \neq \emptyset, \forall j \in \Z^d, \delta \norm{i-j}\leq L\}, \\
		\Omega_{i,L}\Ssup{0} := \{ \omega \in \Omega, \omega_{T_i} = \emptyset, \forall j \in \Z^d, \delta \norm{i-j}\leq L\},
	\end{gather}
	where $\norm{\cdot}$ denotes the Euclidean norm. We assume that there is a measurable function $E_0 : \Omega_f \mapsto \R \cup \{\infty\}$ such that for any $\omega \in \Omega_f$,
	\begin{equation}
		H(\omega) = \sum_{i \in \Z^d}E_i(\omega),
	\end{equation}
	where $E_i (\omega) = E_0(\omega - \delta i)$. $E_i$ is the energy assigned to the tile $T_i$. In order to simplify the notations, for any $\Lambda \in \Z^d$, the total energy assigned on $\hat \Lambda$ is denoted by
	\begin{equation}
		E_\Lambda = \sum_{i \in \Lambda} E_i.
	\end{equation}
	Note that such a measurable function $E_0$ might not be unique. We assume that the energy assigned to $T_0$, $E_0$, satisfies the following set of assumptions :
	\begin{enumerate}[($\Ecal$1)]\label{E}
		\item {\bf Non-degenerate}. $E_0(\emptyset) = 0$. \label{E:non-degenerate}
		\item {\bf Local function}. There exists $r>0$ such that for any configuration $\omega \in \Omega$, $$E_0(\omega) = E_0(\omega_{T_0 \oplus B(0,r)}).$$\label{E:localF}
		\item {\bf Stable}. There exists $c_s \geq 0$ such that for any configuration $\omega \in  \Omega$, $$E_0(\omega) \geq -c_s (1 + N_{T_0 \oplus B(0,r)}(\omega)).$$ \label{E:stable}
		\item {\bf Tame function}. There exists $c_t \geq 0$ such that for any configuration $\omega \in  \Omega$, $$E_0(\omega) \leq c_t (1 + N_{T_0 \oplus B(0,r)}^2(\omega)).$$\label{E:tame}
		\item {\bf Saturation}. There exists $L>0$, $\delta>0$, $\be \in \R$ and $\be_0 \in \R$  such that for any homogeneous configurations around $T_0$, $$\forall \omega \in \Omega_{i,L}\Ssup{0} \cup \Omega_{i,L}\Ssup{1}, \quad E_0(\omega) = \overline{E}_0(\omega) := \be N_{T_0}(\omega) + \be_0 \1_{\omega_{T_0} \neq \emptyset}.$$ \label{E:saturation}
	\end{enumerate}
	The non-degeneracy assumption \ref{E:non-degenerate}, is required to ensure a non-degenerate Hamiltonian \ref{H:nondegenerate}. If we assume \ref{E:localF} then the Hamiltonian is finite range, \ref{H:finiterange}, with a range of $r+2\sqrt{d}\delta$. Furthermore, if we assume \ref{E:stable} then the Hamiltonian is stable, \ref{H:stability}. The assumption \ref{E:tame} is a technical assumption needed in order to prove Lemma~\ref{lem: pression_inde_bord}. The saturation property, \ref{E:saturation}, is one of the main ingredients needed to prove the first-order phase transition. Here we assume that in the setting where the configuration is homogeneous the energy depends only on the number of points inside the tiles. One can show that if $\overline{E}_0$ depends only on the number of points in $T_0$, then it must take the form described in \ref{E:saturation}, other expression is not possible. Now let us provide examples of Hamiltonian that satisfy these assumptions and especially the saturation property. 
	
	\begin{enumerate}
		\item {\bf K-nearest neighbour Strauss interaction}. It is an example of pairwise interaction with nearest neighbours \cite{geyer1999likelihood}. Let $i \in \N$, given a configuration $\omega \in \Omega_f$, a point $y \in \omega$ is the $i$-th neighbour of $x \in \omega$ if $\# B(x,|x-y|) \cap \omega = i$ and we denote by $y = v_i(x, \omega)$. If there are several points at the same distance to $x$ we can list these points using the lexicographical order on the cartesian coordinate. We consider $K \geq 1$, $R>0$ and $A\geq 0$, for any configuration $\omega \in \Omega_f$ the Hamiltonian is given by,
		\begin{equation}
			 H(\omega) = \begin{cases}
				\sum_{x \in \omega} \sum_{i=1}^{\min(K, N(\omega)-1)} A \1_{|x-v_i(x,\omega)| \leq R} & \text{if } N(\omega) \geq 2, \\
				0 & \text{otherwise}.
			\end{cases}
		\end{equation}
		For this interaction, a natural way to assign energy to each tile is, for any configuration $\omega \in \Omega_f$,
		\begin{equation}
			E_0(\omega) = \begin{cases}
				\sum_{x \in \omega_{T_0}} \sum_{i=1}^{\min(K, N(\omega)-1)} A \1_{|x-v_i(x,\omega)| \leq R} & \text{if } N(\omega) \geq 2, \\
				0 & \text{otherwise}.
				\end{cases}
		\end{equation}
		By definition, $E_0$ satisfies \ref{E:non-degenerate} and \refeq{E:localF}. Furthermore, for any configuration $\omega \in \Omega_f$ we have
		\begin{equation}
			0 \leq E_0(\omega) \leq A K N(\omega_{T_0}).
		\end{equation}
		Therefore, $E_0$ satisfies \ref{E:stable} and \ref{E:tame}. We can fix $L=R$ and $\delta$ small enough such that for any $x \in T_0, \# \{i \in \Z^d, T_i \subset B(x,R) \} \geq K+1$. As a consequence, for any homogeneous configurations around $T_0$, $\omega \in \Omega_{i,L}\Ssup{0} \cup \Omega_{i,L}\Ssup{1}$, 
		\begin{equation}
			E_0 (\omega) = A K N(\omega_{T_0}) =: \overline{E}_0(\omega),
		\end{equation}
		and thus $E_0$ satisfies the saturation property, \ref{E:saturation}.
		\item {\bf Quermass interaction with random bounded radii}. It is a classical morphologic interaction \cite{Kendall}. We consider the mark space to be $M = [R_1, R_2]$ and let $\theta_i \in \R$ for $i \in [0,d]$. For any configuration, $\omega \in \Omega_f$, the Hamiltonian is given by,
		\begin{equation}
			H(\omega) = \sum_{i =0}^d \theta_i M_i^d(L(\omega)),
		\end{equation} 
		where $M_i^d$ is the $i$-th Minkowski functional and $L(\omega)$ is the halo of the configuration, i.e.
		\begin{equation}
			L(\omega) = \bigcup_{(x,r) \in \omega} B(x,r).
		\end{equation}
		In particular, $M_d^d$ is the Lebesgue measure, $M_{d-1}^d$ is the $d_1$ Hausdorff measure of the surface and $M_0^d$ is the Euler-Poincaré characteristic. When $\theta_i = 0$ for $i = 0, \dots, d-1$ and $\theta_d > 0$, the Quermass interaction corresponds to the one component Widom-Rowlinson model, \cite{WidomRowlinson}. In dimension $d \leq 2$, the Quermass interaction satisfies our assumption on the Hamiltonian, especially the stability assumption \ref{H:stability}, \cite{Kendall}. In higher dimension $(d\geq 3)$, the Quermass interaction might not be stable, see \cite{Kendall} for further details. 
		
		We call a facet $F$ any non-empty intersection of closed tiles $\left(\overline{T}_i\right)_{i \in \Z^d}$. We denote by $\Fcal$ the set of all facets and $\Fcal_0 := \{ F \in \Fcal, F \cap T_0 \neq \emptyset\}$. Therefore, by additivity property of the Minkowski functionals the natural way to assign energy to each tile is, for any configuration $\omega \in \Omega_f$,
		\begin{equation}
			E_0(\omega) = \sum_{i=0}^d \theta_i M_{i,0}^d(L(\omega)), 
		\end{equation}
		where for any subset $A \in \R^d$,
		\begin{equation*}
			M_{i,0}^d(A) = \sum_{k=i}^d  \sum_{\substack{F \in \mathcal{F}_0 \\ \dim(F)=k}} (-1)^{d-k} M_i^d(A \cap F).
		\end{equation*}
		By definition $E_0$ satisfies \ref{E:non-degenerate} and since the radii of each particle is bounded it also satisfies \ref{E:localF}. In order to verify \ref{E:stable} and \ref{E:tame} we will consider that $\theta_i=0$ for $i=0,\dots, d-2$ when $d\geq3$ and for $\theta_i \in \R$ for $i=0,\dots,d$ when $d\leq 2$. Indeed, at any dimension we have that for any facet $F \in \Fcal_0$ and any configuration $\omega \in \Omega_f$ we have,
		\begin{gather}
			0 \leq M_{d,0}^d(L(\omega ) \cap F) \leq N_{T_0 \oplus  B(0,R_2)} M_{d}^d(T_0), \\
			0\leq M_{d-1,0}^d(L(\omega ) \cap F) \leq N_{T_0 \oplus B(0,R_2)}(\omega) M_{d-1}^d (B(0,R_2)).
		\end{gather}
		Now we consider that we are in dimension 2 and look at the contribution of the Euler-Poincaré characteristic to the energy of a tile. We have for facets of dimension 0, $M_{0}^2(L(\omega) \cap F ) \in \{0,1\}$. Then for facets of dimension 1, we have that $L(\omega) \cap F$ has no holes and therefore $ M_{0}^2(L(\omega) \cap F ) = N_{cc} (L(\omega) \cap F) \leq N_{T_0 \oplus B(0,R_2)}(\omega)$, where $N_{cc}$ counts the number of connected components. For the only facet of dimension 2, $\overline{T}_0$, we know that the number of holes is at most linear, \cite{Kendall}. Therefore, there is a constant $C_0>0$ such that 
		\begin{equation}
			|M_{0,0}^2(\omega)| \leq C_0 (1 + N_{T_0 \oplus B(0,R_2)}(\omega)).
		\end{equation}
		The same bound can be obtained for $M_{0,0}^1$ in dimension 1. As a consequence, under these restrictions the Quermass interaction satisfies \ref{E:stable} and \ref{E:tame}. Finally, if we fix $\delta \leq \frac{R_1}{\sqrt{d}}$ and $L>R_2$, we have that for any homogeneous configuration, $\omega \in \Omega_{i,L}\Ssup{0} \cup \Omega_{i,L}\Ssup{1}$,
		\begin{equation}
			E_0(\omega) = \delta^d \1_{\omega_{T_0} \neq \emptyset} =: \overline{E}_0(\omega),
		\end{equation}
		and therefore the Quermass interaction satisfies the saturation property \ref{E:saturation}.
 		\item {\bf Diluted pairwise interaction}. This example is new and introduced in this paper to provide an saturated approximation of the pairwise interaction which is not saturated. Let $R>0$ and $\phi$ be a radial pair potential such that $r^{d-1}\phi \in L^1$ and has a compact support. We define $R_1 = \sup \{ r \in \R_+,\phi(r)>0\}$ and $R_2 >0$ such that for $r>R_2$, $\phi(r)=0$.  For any configuration, $\omega \in \Omega_f$, the Hamiltonian is given by,
 		\begin{equation}
 			H(\omega) = \iint_{L_R(\omega)^2} \phi(|x-y|) \d x \d y,
 		\end{equation}
 		where
 		\begin{equation*}
 			L_R(\omega) = \bigcup_{x \in \omega} B(x,R).
 		\end{equation*}
 		This Hamiltonian is an approximation of the usual pairwise interaction. Indeed, when rescaled properly and having scale of dilution, $R$, going to $0$ we obtain the pairwise interaction. The way we assign energy to each tile is, for any configuration $\omega \in \Omega_f$,
 		\begin{equation}
 			E_0(\omega) = \int_{L_R(\omega) \cap T_0} \int_{L_R(\omega)} \phi(|x-y|) \d x \d y.
 		\end{equation}
 		By definition $E_0$ satisfies \ref{E:non-degenerate}. Since $\phi$ has a compact support therefore $E_0$ satisfies \ref{E:localF}. We have for any configuration $\omega \in \Omega_f$,
 		\begin{equation}
 			|E_0(\omega)| \leq  \delta^d \int_{\R^d} \left|\phi(|x|)\right| \d x,
 		\end{equation}
 		and therefore $E_0$ satisfies \ref{E:stable} and \ref{E:tame}. Finally, if we fix $\delta \leq \frac{R}{\sqrt{d}}$ and $L > R_2+ 2\sqrt{d}\delta$, we have that for any homogeneous configuration, $\omega \in \Omega_{i,L}\Ssup{0} \cup \Omega_{i,L}\Ssup{1}$,
 		\begin{equation}
 			E_0(\omega) = \delta^d C_\phi \1_{\omega_{T_0} \neq \emptyset} =: \overline{E}_0(\omega) \quad \text{where } C_\phi = \int_{\R^d} \phi(|x|) \d x,
 		\end{equation}
 		and therefore diluted pairwise interaction satisfies the saturation property \ref{E:saturation}.
	\end{enumerate}

	\subsection{Gibbs marked point process}

	Now we introduce the concept of (infinite-volume) Gibbs measures.
	\begin{definition}
		A probability measure $P$ on $\Omega$ is a Gibbs measure for the Hamiltonian $H$, the activity $z >0$ and the inverse temperature $\beta \geq 0$ if $P$ is stationary in space and if for any $\Delta \in  \mathcal{B}_b(\R^d)$ and any bounded positive measurable function $f : \Omega \rightarrow \R$,
		\begin{equation}\label{eq_dlr}
			\int f(\omega) P(\d\omega) = \int \int \frac{1}{Z_\Delta^{z,\beta}(\omega_{\Delta^c})} f(\omega_\Delta' \cup \omega_{\Delta^c}) e^{-\beta H_\Delta(\omega_\Delta' \cup \omega_{\Delta^c})} \Pi_\Delta^z(\d\omega'_\Delta) P(\d\omega)
		\end{equation} 
		where $Z_\Delta^{z,\beta}(\omega_{\Delta^c})$ is the partition function given the outer configuration $\omega_{\Delta^c}$
		\begin{equation*}
			Z_\Delta^{z,\beta}(\omega_{\Delta^c}) = \int e^{-\beta H_\Delta(\omega_\Delta' \cup \omega_{\Delta^c})} \Pi_\Delta^z(\d\omega').
		\end{equation*}
		We denote by $\Gcal(H, z, \beta)$ the set of all Gibbs measures. 
	\end{definition}
	The equations \ref{eq_dlr} are to referred as the DLR equations (for Dobrushin, Lanford and Ruelle) and describe the system at equilibrium. The existence and uniqueness of Gibbs measures are challenging and natural questions. Under the assumptions \ref{H:nondegenerate}-\ref{H:finiterange} the existence of such Gibbs measures are ensured by \cite[Theorem 1]{minicours} or by \cite{roelly2020marked}. Non-uniqueness of Gibbs measures at some parameter $(z, \beta)$ exhibits the coexistence of several state of matter. Our strategy to prove non-uniqueness is, for some set of parameters $(z, \beta) \in \R_+^2$, to exhibit two Gibbs measures $\P^+, \P^- \in  \Gcal(H, z, \beta)$ such that
	\begin{equation}
		\rho(\P^+) := E_{\P^+}(N_{[0,1]^d}) > \rho(\P^-).
	\end{equation}
	We call this phenomenon a first-order phase transition. The measure of a jump in density between two pure phases is effective in detecting a liquid-gas phase transition. However it is not enough as one needs to prove that crystallisation and symmetry breaking does not occur. In relation with the Gibbs measures, the pressure, denoted by $\psi$, is a quantity of interest and it is defined as 
		\begin{equation}
		\psi (z, \beta) : = \lim\limits_{n \rightarrow +\infty} \frac{1}{|\Delta_n|} \ln(Z_{\Delta_n}^{z,\beta}),
	\end{equation}
	where  $(\Delta_n)$ is the sequence of the following boxes $[-n,n]^d$ and  $Z_{\Delta_n}^{z,\beta}$ is the partition function with free boundary condition (i.e.  $Z_{\Delta_n}^{z,\beta}(\emptyset)$). Under the assumption of finite range interaction, this limit always exists as a consequence of sub-additivity of the sequence $(\ln Z_{\Delta_n}^{z,\beta})_{n\in \N}$, see \cite[Lemma 1]{dereudre2016variational} for more detail. The regularity of the pressure is related to the phase transition phenomenon. In particular the first order phase transition is related to its non-differentiability for some parameters. As a result of the Pirogov-Sinaï-Zahradn\'ik method, we will prove that the pressure is non-differentiable for a critical activity and $\beta$ large enough. 
	\subsection{Contours}
	In order to state our main results, , it is necessary to introduce the notion of contours, as it plays a crucial role in the statement of our results. We define the following application 
	\begin{align*}
		\sigma :& \Omega_f \times \Z^d \rightarrow \{ 0, 1\} \\
		& (\omega, i) \mapsto \begin{cases}
			0 & \text{if } \omega_{T_i} =\emptyset \\
			1 & \text{otherwise}
		\end{cases}.
	\end{align*}
	In the following we use the notation $\sharp$ for either $0$ or $1$. Let $L>0$ and $\omega \in \Omega$, a site $i \in \Z^d$ is said to be $\sharp$-correct if for all sites $j$ such that $\delta \norm{i-j} \leq 2L$, we have $\sigma(\omega, j) = \sharp$. A site $i$ is non-correct when it fails to be $\sharp$-correct for any $\sharp \in \{0,1\}$. The set of all non-correct sites is denoted by $\overline{\Gamma}(\omega) $. We can partition $\overline{\Gamma}(\omega)$ into its maximum connected components that we denote by $\overline{\gamma}(\omega)$ and we call it the support of a contour. We define a contour as the pair $\gamma(\omega) := (\overline{\gamma}(\omega), (\sigma(\omega,i))_{i \in \overline{\gamma}(\omega)})$ and we denote by $\Gamma(\omega)$ the set of all contours that appear with the configuration $\omega$.
	\newline
	
	Furthermore, the number of connected components is finite and for any $\overline{\gamma}(\omega)$, since we are considering only finite configurations, the complementary set has a finite amount of maximum connected components and in particular we have only one unbounded connected component and we call it the exterior of $\overline{\gamma}(\omega)$ that we denote by $ext(\overline{\gamma}(\omega))$. A contour $\gamma(\omega)$ is said to be external when for any other contour $\gamma'(\omega)$, $\overline{\gamma}(\omega) \subset ext(\overline{\gamma}'(\omega))$ and we denote by $\Gamma_{ext}(\omega)$ the subset of $\Gamma(\omega)$ comprised of only external contours. Let $\Lambda \subset \Z^d$ and $L>0$, we define the exterior boundary~$\partial_{ext} \Lambda$ and the interior boundary~$\partial_{int} \Lambda$ of $\Lambda$ as
	\begin{align*}
		\partial_{ext} \Lambda = \{j \in \Lambda^c, \delta d_2(j, \Lambda) \leq 2L \}  \\
		\partial_{int} \Lambda = \{i \in \Lambda, \delta d_2(i, \Lambda^c) \leq 2L+\delta \},
	\end{align*}
	where $d_2$ is the Euclidean distance in $\R^d$. Thus for any  $\gamma(\omega)$ and any $A$ maximum connected component of $\overline{\gamma}(\omega)^c$, according to \cite[Lemma 7.23]{velenik} there is a unique $\sharp \in \{0,1\}$ such that for any $i \in \partial_{ext} A \cup \partial_{int}A, \sigma(\omega, i) = \sharp$ and we define $\Label(A) := \sharp$. Furthermore, we define the type of a contour $\gamma(\omega)$ as $\Type(\gamma(\omega)) := \Label(ext(\overline{\gamma}(\omega)))$. Finally, we call the interiors of a contour the sets
	\begin{equation*}
		\Int_\sharp \gamma(\omega) = \bigcup_{\substack{A \neq ext(\overline{\gamma}(\omega)) \\  \Label(A) = \sharp}} A  \quad  \text{ and } \quad \Int \gamma(\omega) = \Int_0 \gamma(\omega) \cup \Int_1 \gamma(\omega).
	\end{equation*}
	For $k \in \N$, we say that a contour $\gamma(\omega)$ is of class $k$ if $|\Int \gamma(\omega)| = k$. 
	\newline
	
	Up until now, we have considered only configurations that are achievable by a configuration, later we need the concept of abstract contours. We define an abstract contour as $\Gamma := \{ \gamma_i = (\overline{\gamma}_i, (\sharp_j)_{j \in \overline{\gamma_i}}) , i \in I  \subset \N\}$ for which  each contour $\gamma_i$ is achievable for some configuration $\omega_i$. We do not assume the global achievability. Moreover, $\Gamma$ is said to be geometrically compatible if for all $\{i,j\} \subset I$, $d_\infty(\overline{\gamma}_i, \overline{\gamma}_j) >1$, where $d_\infty$ is the infinite distance. For any $\Lambda \subset \Z^d$, we denote by $\Ccal\Ssup{\sharp}(\Lambda)$ the collection of geometrically compatible sets of contours of the type $\sharp$ such that $d_\infty(\overline{\gamma}_i, \Lambda^c)>1$. We allow the set $\Gamma = \{ (\emptyset, \emptyset)\}$ to belong to any collection $\Ccal\Ssup{\sharp}(\Lambda)$, it corresponds to the event where not a single contour appears in $\Lambda$. Furthermore, we denote by $\Ccal\Ssup{\sharp}_{ext}(\Lambda)$ the sub-collection of sets $\Gamma$ containing only external contours and by $\Ccal\Ssup{\sharp}_n(\Lambda)$ the sub-collection where for any $\gamma \in \Gamma$ we have $|\Int \gamma| \leq n$. 
	
	\begin{figure}[H]
		
		\subfloat[] {
			\begin{tikzpicture}[scale=0.8]\label{fig_a}
				\def\grillesize{21} 
				\def\cellsize{0.4cm}   
				
				\foreach \x/\y in {3/3, 3/4, 3/5, 4/6, 4/7, 3/8,
					5/9, 5/10, 5/11, 6/12, 5/13, 4/14, 4/15, 3/16, 3/17, 5/18, 4/18, 4/17, 4/16, 5/15, 5/14, 4/13, 5/12, 4/11, 4/10, 4/9, 4/8, 5/7, 5/6, 4/5, 4/4, 4/3, 4/2, 5/2, 5/3, 6/2, 6/3, 7/3, 7/4, 8/3, 8/4, 9/3, 9/4, 10/2, 10/3, 11/2, 12/2, 11/3, 12/3, 12/4, 13/3, 13/4, 14/3, 14/4, 15/3, 15/4, 16/2, 16/3, 17/2, 17/3, 17/4, 18/3, 18/4, 18/5, 17/5, 18/6, 17/6, 18/7, 17/7, 16/8, 17/8, 16/9, 17/9, 16/10, 17/10, 17/11, 18/11, 17/12, 18/12, 17/13, 18/13, 16/14, 17/14, 16/15, 17/15, 17/16, 18/16, 17/17, 18/17, 5/17, 6/17, 6/18, 7/16, 7/17, 8/16, 8/17, 9/16, 9/17, 10/17, 10/18, 11/17, 11/18, 12/17, 12/18, 13/16, 13/17, 14/16, 14/17, 15/17, 15/18, 16/16, 16/17, 11/9, 11/10, 11/11, 12/10, 12/11, 10/10}{
					\fill[gray!50] 
					({\x*\cellsize}, {\y*\cellsize})
					rectangle 
					({(\x+1)*\cellsize}, {(\y+1)*\cellsize});
				}
				
				\foreach \x/\y in {2/2, 3/2, 3/1, 4/1, 5/1, 6/1, 7/1, 7/2, 8/2, 9/2, 9/1, 10/1, 11/1, 12/1, 13/1, 13/2, 14/2, 15/2, 15/1, 16/1, 17/1, 18/1, 18/2, 19/2, 19/3, 19/4, 19/5, 19/6, 19/7, 19/8, 18/8, 18/9, 18/10, 19/10, 19/11, 19/12, 19/13, 19/14, 18/14, 18/15, 19/15, 19/16, 19/17, 19/18, 18/18, 17/18, 16/18, 16/19, 15/19, 14/19, 14/18, 13/18, 13/19, 12/19, 11/19, 10/19, 9/19, 9/18, 8/18, 7/18, 7/19, 6/19, 5/19, 4/19, 3/19, 3/18, 2/18, 2/17, 2/16, 2/15, 3/15, 3/14, 3/13, 3/12, 4/12, 3/11, 3/10, 3/9, 2/9, 2/8, 2/7, 3/7, 3/6, 2/6, 2/5, 2/4, 2/3 } {
					\fill[blue!50] 
					({\x*\cellsize}, {\y*\cellsize})
					rectangle 
					({(\x+1)*\cellsize}, {(\y+1)*\cellsize});
				}
				
				\foreach \x/\y in {6/4, 5/4, 5/5, 6/5, 6/6, 6/7, 6/8, 5/8, 6/9, 6/10, 6/11, 7/11, 7/12, 7/13, 6/13, 6/14, 6/15, 6/16, 5/16, 7/15, 8/15, 9/15, 10/15, 10/16, 11/16, 12/16, 12/15, 13/15, 14/15, 15/15, 15/16, 15/14, 15/13, 16/13, 16/12, 16/11, 15/11, 15/10, 15/9, 15/8, 15/7, 16/7, 16/6, 16/5, 16/4, 15/5, 14/5, 13/5, 12/5, 11/5, 11/4, 10/4, 10/5, 9/5, 8/5, 7/5} {
					\fill[red!50] 
					({\x*\cellsize}, {\y*\cellsize})
					rectangle 
					({(\x+1)*\cellsize}, {(\y+1)*\cellsize});
				}
				
				\foreach \x/\y in {10/9, 9/9, 9/10, 9/11, 10/11, 10/12, 11/12, 12/12, 13/12, 13/11, 13/10, 13/9, 12/9, 12/8, 11/8, 10/8} {
					\fill[red!50] 
					({\x*\cellsize}, {\y*\cellsize})
					rectangle 
					({(\x+1)*\cellsize}, {(\y+1)*\cellsize});
				}
				\foreach \x in {0,1,...,\grillesize} {
					\draw[black, thin] (0, \x*\cellsize) -- (\grillesize*\cellsize, \x*\cellsize);
					\draw[black, thin] (\x*\cellsize, 0) -- (\x*\cellsize, \grillesize*\cellsize);
				}
				\node[ font=\bfseries\Large] at (11.5*\cellsize,10.5*\cellsize) {$\gamma_2$};
				
				\node[font=\bfseries\Large] at (11.5*\cellsize,17.5*\cellsize) {$\gamma_1$};
				
				\node[font=\bfseries\Large] at (11.5*\cellsize,14.5*\cellsize) {$A$};
			\end{tikzpicture}
		}
		\quad
		\subfloat[] {
			\begin{tikzpicture}[scale=0.8]\label{fig_b}
				\def\grillesize{21} 
				\def\cellsize{0.4cm}   
				
				\foreach \x/\y in { 3/3, 3/4, 3/5, 4/6, 4/7, 3/8, 5/9, 5/10, 5/11, 6/12, 5/13, 4/14, 4/15, 3/16, 3/17, 5/18, 4/18, 4/17, 4/16, 5/15, 5/14, 4/13, 5/12, 4/11, 4/10, 4/9, 4/8, 5/7, 5/6, 4/5, 4/4, 4/3, 4/2, 5/2, 5/3, 6/2, 6/3, 7/3, 7/4, 8/3, 8/4, 9/3, 9/4, 10/2, 10/3, 11/2, 11/3, 12/2, 12/3, 12/4, 13/3, 13/4, 14/3, 14/4, 15/3, 15/4, 16/2, 16/3, 17/2, 17/3, 17/4, 18/3, 18/4, 18/5, 17/5, 18/6, 17/6, 18/7, 17/7, 16/8, 17/8, 16/9, 17/9, 16/10, 17/10, 17/11, 18/11, 17/12, 18/12, 17/13, 18/13, 16/14, 17/14, 16/15, 17/15, 17/16, 18/16, 17/17, 18/17, 5/17, 6/17, 6/18, 7/16, 7/17, 8/16, 8/17, 9/16, 9/17, 10/17, 10/18, 11/17, 11/18, 12/17, 12/18, 13/16, 13/17, 14/16, 14/17, 15/17, 15/18, 16/16, 16/17, 11/9, 11/10, 11/11, 12/10, 12/11, 10/10} {
					\fill[gray!50] 
					({\x*\cellsize}, {\y*\cellsize})
					rectangle 
					({(\x+1)*\cellsize}, {(\y+1)*\cellsize});
				}
				
				\foreach \x/\y in {2/2, 3/2, 3/1, 4/1, 5/1, 6/1, 7/1, 7/2, 8/2, 9/2, 9/1, 10/1, 11/1, 12/1, 13/1, 13/2, 14/2, 15/2, 15/1, 16/1, 17/1, 18/1, 18/2, 19/2, 19/3, 19/4, 19/5, 19/6, 19/7, 19/8, 18/8, 18/9, 18/10, 19/10, 19/11, 19/12, 19/13, 19/14, 18/14, 18/15, 19/15, 19/16, 19/17, 19/18, 18/18, 17/18, 16/18, 16/19, 15/19, 14/19, 14/18, 13/18, 13/19, 12/19, 11/19, 10/19, 9/19, 9/18, 8/18, 7/18, 7/19, 6/19, 5/19, 4/19, 3/19, 3/18, 2/18, 2/17, 2/16, 2/15, 3/15, 3/14, 3/13, 3/12, 4/12, 3/11, 3/10, 3/9, 2/9, 2/8, 2/7, 3/7, 3/6, 2/6, 2/5, 2/4, 2/3 } {
					\fill[blue!50]
					({\x*\cellsize}, {\y*\cellsize})
					rectangle 
					({(\x+1)*\cellsize}, {(\y+1)*\cellsize});
				}
				
				\foreach \x/\y in {6/4, 5/4, 5/5, 6/5, 6/6, 6/7, 6/8, 5/8, 6/9, 6/10, 6/11, 7/11, 7/12, 7/13, 6/13, 6/14, 6/15, 6/16, 5/16, 7/15, 8/15, 9/15, 10/15, 10/16, 11/16, 12/16, 12/15, 13/15, 14/15, 15/15, 15/16, 15/14, 15/13, 16/13, 16/12, 16/11, 15/11, 15/10, 15/9, 15/8, 15/7, 16/7, 16/6, 16/5, 16/4, 15/5, 14/5, 13/5, 12/5, 11/4, 10/4, 11/5, 10/5, 9/5, 8/5, 7/5} {
					\fill[red!50] 
					({\x*\cellsize}, {\y*\cellsize})
					rectangle 
					({(\x+1)*\cellsize}, {(\y+1)*\cellsize});
				}
				
				\foreach \x/\y in {10/9, 9/9, 9/10, 9/11, 10/11, 10/12, 11/12, 12/12, 13/12, 13/11, 13/10, 13/9, 12/9, 12/8, 11/8, 10/8} {
					\fill[blue!50] 
					({\x*\cellsize}, {\y*\cellsize})
					rectangle 
					({(\x+1)*\cellsize}, {(\y+1)*\cellsize});
				}
				\foreach \x in {0,1,...,\grillesize} {
					\draw[black, thin] (0, \x*\cellsize) -- (\grillesize*\cellsize, \x*\cellsize);
					\draw[black, thin] (\x*\cellsize, 0) -- (\x*\cellsize, \grillesize*\cellsize);
				}
				\node[ font=\bfseries\Large] at (11.5*\cellsize,10.5*\cellsize) {$\gamma_2$};
				
				\node[font=\bfseries\Large] at (11.5*\cellsize,17.5*\cellsize) {$\gamma_1$};
				
				\node[font=\bfseries\Large] at (11.5*\cellsize,14.5*\cellsize) {$A$};
				
			\end{tikzpicture} 
		}
		\caption{The contour corresponds to the grey areas, while the blue and red squares represent the tiles at the boundary of the contour where the spins are $\sharp$ and $1-\sharp$, respectively. In Figure \ref{fig_a}, the contour $\Gamma =\{\gamma_1, \gamma_2\}$ is achievable by some configuration $\omega$ because the label of $A$ matches for both $\gamma_1$ and $\gamma_2$. In contrast, in Figure \ref{fig_b}, the contour $\Gamma = \{\gamma_1, \gamma_2\}$ is not globally achievable by any configuration. In this case, $\Gamma \in \mathcal{C}\Ssup{\sharp}(\Lambda)$, where the types of $\gamma_1$ and $\gamma_2$ are the same, but the labels of $A$ are mismatched. (reproduced from \cite{renaudchanQuermass})}
	\end{figure}
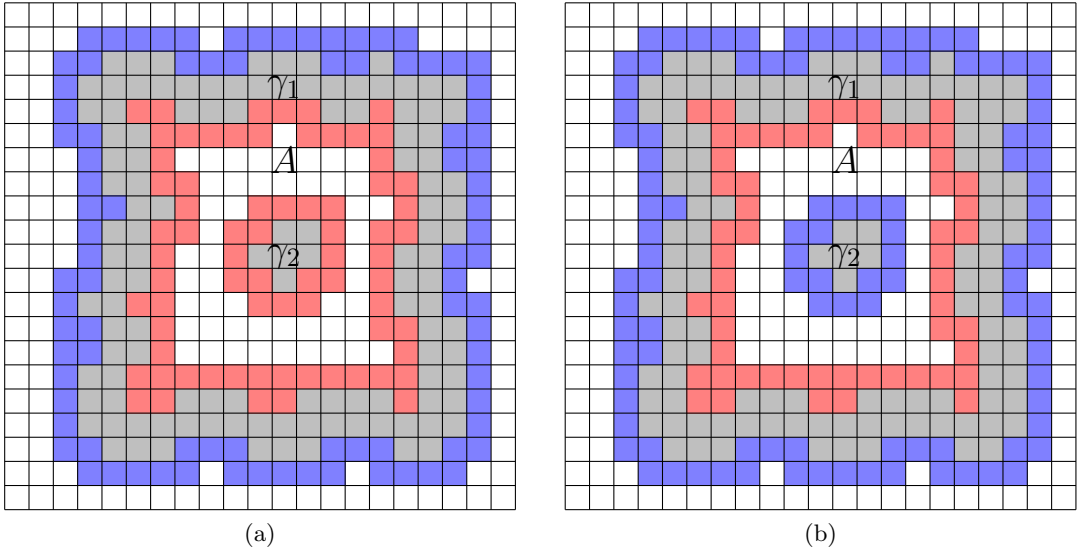
	
	\section{Results}\label{Section3}
	We now present our results on phase transition for saturated interactions. We begin with a general result on first-order phase transitions for saturated interactions for which the Hamiltonian satisfies a Peierls condition. The proof of this result will be given in the following section. Afterwards, we explore methods to verify a Peierls condition and some examples of phase transitions that can be obtained using this approach. 
	
	\begin{theorem}\label{thm: liquid_gaz1}
		Let $H$ be a Hamiltonian that satisfies \ref{H:stability}-\ref{H:finiterange} such that $E_0$  exists and satisfies \ref{E:non-degenerate}-\ref{E:saturation}. Furthermore, we suppose that the interaction satisfies a Peierls-like condition, i.e. there is $\be_+>0$ such that for any contour $\gamma$ and any configuration $\omega$ that achieves this contour, we have
		\begin{align}\label{cond: peierls_thm}
			E_{\overline{\gamma}} - \overline{E}_{\overline{\gamma}}(\omega) \geq \be_+ |\overline{\gamma}|.
		\end{align}		
		 For any $\beta>0$, we fix $z_\beta^-$ and $z_\beta^+$ as
		\begin{align}\label{def: z_beta_pm}
			z_\beta^- := \frac{e^{\beta \be}}{\delta^d} \ln\left(1 + e^{\beta \be_0 - 2}\right), \quad z_\beta^+ := \frac{e^{\beta \be}}{\delta^d} \ln\left(1 + e^{\beta \be_0 + 2}\right)
		\end{align} 
		and the open interval $O_\beta := (z_\beta^-, z_\beta^+)$. Then there is $\beta_c>0$ such that for $\beta \geq \beta_c$ there exists $z_\beta^c \in O_\beta$ for which a first-order phase transition occurs. More specifically, we have
		\begin{equation}\label{saut_de_dérivée_pression_au_point_critique}
			\frac{\partial \psi}{\partial z^+}(\beta, z_\beta^c) > \frac{\partial \psi}{\partial z^-}(\beta, z_\beta^c),
		\end{equation}		
		and we can find two Gibbs measures $\P^+, \P^- \in \mathcal{G}(H,z_\beta^c, \beta)$ such that
		\begin{equation}
			\rho(\P^+) = z + z \frac{\partial \psi}{\partial z^+}(z_\beta^c, \beta) \quad \text{and } \quad \rho(\P^-) = z + z\frac{\partial \psi}{\partial z^-}(z_\beta^c, \beta).
		\end{equation}
	\end{theorem}
	The proof of this result is given in Section \ref{sec: proof_thm1} and is inspired by the techniques developed by Pirogov, Sinaï and Zahradn\'ik for lattice systems. In addition to the saturation property, the Peierls condition is one of the main ingredients for proving a phase transition. However, this condition is often difficult to establish. We provide an easier way to verify it using what we call dominoes, which, given a contour, are couples of adjacent tiles with different occupation status. For any contour $\gamma$, we define the set of dominoes as 
	\begin{equation}\label{def: dominoes}
		D(\gamma) := \left\{ (i,j) \in \overline{\gamma}^2, \norm{i-j}_\infty = 1, \sharp_i = 1, \sharp_j = 0 \right\}.
	\end{equation}
	We know by \cite[Lemma 5]{renaudchanQuermass} that there is $r_D > 0$  such that for any contour $\gamma$ we have
	\begin{equation}
		|D(\gamma)| \geq r_D |\overline{\gamma}|.
	\end{equation}
	Therefore, if we assume that for any configuration $\omega \in \Omega_f$ we have $E_0(\omega) \geq \overline{E}_0(\omega)$ and that for any contour $\gamma(\omega)$ and any domino $(i,j) \in D(\gamma(\omega))$, we have $E_j(\omega) \geq \be_\emptyset >0$, we can derive the following inequality,
	\begin{equation}
		E_{\overline{\gamma}} - \overline{E}_{\overline{\gamma}}(\omega) \geq \be_\emptyset r_D |\overline{\gamma}|
	\end{equation} 
	and therefore, we obtain the following corollary where the Peierls condition is established through a tile-by-tile analysis of the energy.  We call this phenomenon the energy from the vacuum.
	\begin{corollary}
		Let $H$ be a Hamiltonian which satisfies \ref{H:stability}-\ref{H:finiterange} such that it has $E_0$ that satisfies \ref{E:non-degenerate}-\ref{E:saturation}. Furthermore, we suppose that for any configuration $\omega \in \Omega_f$ we have $E_0(\omega) \geq \overline{E}_0(\omega)$ and that there exists $\be_\emptyset>0$ such that for any contour $\gamma(\omega)$ and any domino $(i,j) \in D(\gamma(\omega))$ we have $E_j(\omega) \geq \be_\emptyset$. Then there is $\beta_c >0$ such that for $\beta \geq \beta_c$ there is $z_\beta^c >0$ for which a first-order phase transition occurs. 
	\end{corollary}
	The saturation property alone is not sufficient to prove a phase transition. We have not been able to prove that the K-nearest neighbour Strauss interaction satisfies a Peierls condition and simulations seem to suggest that no such phenomenon occurs. Building on these general results on first-order phase transitions, we can derive several consequences for specific models, such as the Quermass interaction and the diluted pairwise interaction. For instance, \cite[Theorem 1]{renaudchanQuermass} follows directly from Theorem \ref{thm: liquid_gaz1} once we prove that it satisfies the Peierls condition for a set of parameters. Similarly, we obtain the following first-order phase transition result for the diluted pairwise interaction that we introduced in example 3 in subsection \ref{subsec: def_coarse_sat}. 
	\begin{theorem}\label{thm:fopt_diluted_pair}
		Let $\phi$ be a radial pair potential such that $r^{d-1}\phi \in L^1$ that satisfies
		\begin{equation}\label{cond: phi_global}
			C_d \int\limits_{B(0,R)} \phi^+ \d x > \left[ \left(\frac{R_1}{R} \right)^d-1 \right] \int\limits_{B(0,R_1) \backslash B(0,R)} \phi^+ \d x + \int\limits_{\R^d} \phi^- \d x,
		\end{equation}
		where
		\begin{equation*}
				C_d = \frac{\int_{0}^{\frac{\pi}{3}} \sin(\theta)^{d-2} \d \theta }{\int_{0}^{\pi} \sin(\theta)^{d-2} \d \theta }.
		\end{equation*}
		Then there is $\beta_c>0$ such that for $\beta > \beta_c$, there is $z_\beta^c>0$  for which a first-order phase transition occurs. Furthermore, we know that there is $c>0$ such that 
		\begin{equation}
			\left|z_\beta^c - \beta C_\phi \right| = O(e^{-c \beta}).
		\end{equation} 
		\end{theorem}
		The assumption $\eqref{cond: phi_global}$ is sufficient to obtain a Peierls condition, the proof will be given in Section \ref{sec: proof_thm2}. Furthermore, we can observe that for a pair potential that is repulsive and non-integrable at the origin we can always truncate near the origin such that condition \eqref{cond: phi_global} is verified. Therefore we have the following corollary. 
		\begin{corollary}\label{cor: lgpt cloud de pair troncqué}
			Let $\phi$ be a pair potential with finite support, such that it is non integrable and positive at the origin. For any $R>0$ there is $\epsilon \leq \min\{R_1, R\}$ for which we can build a truncated pair potential
			\begin{align*}
				\phi_\epsilon(r) := \begin{cases}
					\phi(\epsilon) & \text{if } r\leq \epsilon \\
					\phi(r) & \text{otherwise}
				\end{cases},
			\end{align*}  
			such that the diluted pairwise interaction for $\phi_\epsilon$ exhibits a first-order phase transition phenomenon for $\beta > \beta_c(\epsilon)>0$ and a critical activity $z_\beta^c>0$. Furthermore, we know that there is $c>0$ such that 
			\begin{equation}
				\left|z_\beta^c - \beta \int_{\R^d} \phi_\epsilon dy \right| = O(e^{-c \beta}).
			\end{equation} 
		\end{corollary}
		This last corollary opens a new path to study phase transitions of pairwise interaction with strong short-range repulsion, such as the Lennard-Jones potential.

	\section{Proof of Theorem \ref{thm: liquid_gaz1}}\label{sec: proof_thm1}
	In this section, we decompose every step needed to prove the main results presented in this paper and it follows the structure of the proof presented in \cite{renaudchanQuermass}.
	
	\subsection{Existence of Gibbs point processes for different boundary conditions}
	Without loss of generality we may assume $\be = 0$. Indeed, we can replace $H$ by $\tilde{H} := H -\be N$ , which corresponds to a change of activity, $\tilde{z} := z e^{-\beta \be}$. For any $\Lambda \in \Z^d$, we denote by 
	\begin{equation*}
		\partial \Lambda = \{i \in \Lambda, \delta d_2(i, \Lambda^c) \leq L +\delta\} \quad \text{and} \quad \widehat{\Lambda}= \bigcup_{i \in \Lambda} T_i.
	\end{equation*}
	We are building two Gibbs point processes with different boundary conditions, a free boundary and a wired boundary. For any subset $\Lambda \subset \Z^d$, we define the following probability measures
	\begin{equation}
		\P_\Lambda\Ssup{\sharp} = \frac{1}{Z_\Lambda\Ssup{\sharp}} e^{-\beta(E_{\Lambda \setminus \partial\Lambda} + \overline{E}_{\partial\Lambda})} \mathbbm{1}_{(\sharp)_\Lambda} \Pi_{\widehat{\Lambda}}^z,
	\end{equation}
	 where the boundary condition $(\sharp)_\Lambda = \{ \omega \in \Omega, \forall i \in \partial_{int}\Lambda, \sigma(\omega,i) = \sharp \},$ and $Z_\Lambda\Ssup{\sharp}$ is the partition function and it is given by 
	\begin{align}
		Z_\Lambda\Ssup{\sharp}= \int e^{-\beta(E_{\Lambda \setminus \partial\Lambda} + \overline{E}_{\partial\Lambda})} \mathbbm{1}_{(\sharp)_\Lambda} \Pi_{\widehat{\Lambda}}^z. 
	\end{align}
	This probability measure is well defined. Indeed, with the event $E_\Lambda\Ssup{\sharp} = \{\omega \in \Omega, \forall i \in \Lambda,  \sigma(\omega, i)=\sharp\}$ we have 
	\begin{equation}\label{ineq: lower partition function}
		Z_\Lambda\Ssup{\sharp} \geq   e^{-\beta \be_0 \sharp |\Lambda|} \Pi_{\widehat{\Lambda}}^z(E_\Lambda\Ssup{\sharp}) = e^{-(\beta \be_0 \sharp + z \delta^d) |\Lambda|} \left( (1-\sharp) + z \delta^d \sharp \right)^{|\Lambda|} > 0. 
	\end{equation}
	Furthermore, according to \ref{E:saturation} and \ref{E:localF}, for any configuration $\omega \in \Omega_f$ such that $\omega \subset \widehat{\Lambda}$ we have
	\begin{align}
		 \overline{E}_{\partial\Lambda} + E_{\Lambda \setminus \partial\Lambda}(\omega) &\geq \min(0,\be_0) N_{\partial \Lambda}(\omega) +  E_{\Lambda \setminus \partial\Lambda}(\omega) \nonumber \\
		 &\geq \min(0,\be_0) N_{\partial \Lambda}(\omega) - c_S \sum_{i \in \Lambda \setminus \partial\Lambda} 1 + N_{T_i \oplus B(0,r)}(\omega) \nonumber \\
		  &\geq (\min(0,\be_0) -c_S|B(0,r+\sqrt{d}\delta)|) N(\omega) - c_S |\Lambda| . \label{ineq: stab energie}
	\end{align}
	As a consequence, we have that 
	\begin{equation}
		Z_\Lambda\Ssup{\sharp} \leq e^{\beta c_S |\Lambda|} \exp\left(z\delta^d |\Lambda| (e^{- \beta (\min(0,\be_0) -c_S |B(0,R+\sqrt{d}\delta)|)}-1)\right) < \infty.
 	\end{equation}
 	In the following proposition, we prove that $\P_\Lambda\Ssup{\sharp}$ satisfies the DLR equations in the bulk for the Hamiltonian $H$. 
 	
 	\begin{proposition}\label{prop: finite_volume_DLR}
 		Let $\Lambda \subset \Z^d$ and $\Delta \subset \widehat{\Lambda \setminus \partial_{int}\Lambda}$ such that $\lambda(\Delta)>0$. Then for $\P_\Lambda\Ssup{\sharp}$-a.s. all $\omega_{\Delta^c}$
 		\begin{equation}
 			\P_\Lambda\Ssup{\sharp} (\d\eta_\Delta| \omega_{\Delta^c}) = \frac{1}{Z_\Delta(\omega_{\Delta^c})} e^{-\beta H_\Delta(\eta_\Delta \cup \omega_{\Delta^c})} \Pi_\Delta^z(\d\eta),
 		\end{equation}
 		where $Z_\Delta(\omega_{\Delta^c})$ is the normalisation constant given by $Z_\Delta(\omega_{\Delta^c}) = \int e^{-\beta H_\Delta(\eta_\Delta \cup \omega_{\Delta^c})} \Pi_\Delta^z(\d\eta).$
 	\end{proposition}
 	
 	\begin{proof}
 		We denote by $\Lambda_\Delta = \{ i \in \Lambda, T_i \oplus B(0,r) \cap \Delta \neq \emptyset\}$. By definition of $H_\Delta$, we have
 		\begin{align*}
 			H_\Delta (\omega) =  \sum_{i\in \Lambda_\Delta} E_i(\omega) - E_i(\omega_{\Delta^c}).
 		\end{align*}
 		For $i \in \Lambda \setminus \Lambda_\Delta$, by finite range property of $E_0$ we have
 		\begin{align*}
 			E_i(\omega) = E_i (\omega_{\Delta^c}).
 		\end{align*}
 		Therefore, we have
 		\begin{align*}
 			\P_\Lambda\Ssup{\sharp}(\d \omega) &= \frac{1}{Z_\Lambda\Ssup{\sharp}} e^{-\beta (\overline{E}_{\partial \Lambda} + E_{\Lambda \setminus \partial \Lambda})} \mathbbm{1}_{(\sharp)_\Lambda}(\omega) \Pi_{\widehat{\Lambda}}^z(\d\omega)\\
 			&= \frac{1}{Z_\Lambda\Ssup{\sharp}} e^{-\beta H_\Delta(\eta_\Delta \cup \omega_{\Delta^c})} e^{-\beta(\overline{E}_{\partial \Lambda} + E_{\Lambda \setminus (\partial \Lambda \cup \Lambda_\Delta)}-E_{\Lambda_\Delta})(\omega_{\Delta^c})}  \mathbbm{1}_{(\sharp)_\Lambda}(\omega_{\Delta^c}) \Pi_{\Delta}^z(\d\eta) \Pi_{\widehat{\Lambda}\setminus\Delta}^z(\d\omega).
 		\end{align*}
 		As a consequence, the unnormalised conditional density of $\P_\Lambda\Ssup{\sharp}(\d\eta_\Delta |\omega_{\Delta^c})$ with respect to $\Pi_\Delta^z$ is $\eta \mapsto e^{-\beta H_\Delta(\eta_\Delta \cup \omega_{\Delta^c})}$ and with the proposed normalisation we obtain the desired DLR equations for $\Delta \subset \widehat{\Lambda \setminus \partial_{int}\Lambda}$.
 	\end{proof}
 	In the following, we construct two Gibbs measure from the different boundaries condition and they are obtained as limit in the local convergence topology, which is the smallest topology on the space of probability measures on $\Omega$ such that any application $: \P \to \int f \d \P$ is continuous for any local tame function $f$. The measures $\P_\Lambda\Ssup{\sharp}$ are not stationary and therefore not suitable for the local convergence topology. However, for $\Lambda_n = \llbracket -n, n \rrbracket^d$ we can construct empirical fields such that for any test function $f$ we have 
 	\begin{equation}
 		\int f(\omega) \tilde{\P}_{\Lambda_n}\Ssup{\sharp}(\d \omega) = \frac{1}{|\widehat{\Lambda_n}|} \int_{\widehat{\Lambda}_n} \int_\Omega f(\omega + u) \check{\P}_{\Lambda_n}\Ssup{\sharp}(\d \omega) \d u, \; \text{where} \;  \check{\P}_{\Lambda_n}\Ssup{\sharp} = \bigotimes_{i \in \Z^d} \P_{\Lambda_n + 2ni}\Ssup{\sharp}.
 	\end{equation}
 	By construction $\tilde{\P}_{\Lambda_n}\Ssup{\sharp}$ is stationary. 
 	\begin{proposition}\label{prop: champ_empirique_DLR}
 		The empirical field $\left(\tilde{\P}_{\Lambda_n}\Ssup{\sharp}\right)_{n\in\N}$ has an accumulation point, for the local convergence topology, $\P\Ssup{\sharp}$ that is a Gibbs measure for the Hamiltonian $H$, activity $z$ and inverse temperature $\beta$.
 	\end{proposition}
 	
 	\begin{proof}
 		For $\xi >0$ and $\Lambda \subset \Z^d$, we compute the relative entropy of $\P_{\Lambda}\Ssup{\sharp}$ with respect to $\Pi_{\widehat{\Lambda}}^\xi$ we have
 		\begin{align}
 			I(\P_{\Lambda}\Ssup{\sharp} | \Pi_{\widehat{\Lambda}}^\xi) = \E_{\P_{\Lambda}\Ssup{\sharp}} \left(- \beta (E_{\Lambda \setminus \partial\Lambda} + \overline{E}_{\partial\Lambda}) + \ln\frac{z}{\xi} N_{\widehat{\Lambda}}\right) + (\xi - z) \delta^d|\Lambda| - \ln Z_\Lambda\Ssup{\sharp}.
 		\end{align}
 		Therefore, using \eqref{ineq: lower partition function} and \eqref{ineq: stab energie} we have
 		\begin{align}
 				I(\P_{\Lambda}\Ssup{\sharp} | \Pi_{\widehat{\Lambda}}^\xi) \leq \E_{\P_{\Lambda}\Ssup{\sharp}}  \left( (\beta A + \ln\frac{z}{\xi}) N_{\widehat{\Lambda}} \right)  + \left((\xi - z) \delta^d - c_S + \beta \be_0 \sharp - \ln \left((1-\sharp) + z \delta^d \sharp \right) \right)  |\Lambda| 
 		\end{align}
 		where $ A= \max(0,\be_0) + c_S|B(0,r+\sqrt{d}\delta)|$. Therefore, if we fix $\xi = ze^{\beta A}$, we have 
 		\begin{align}
 			\frac{1}{|\Lambda|}  I(\P_{\Lambda}\Ssup{\sharp} | \Pi_{\widehat{\Lambda}}^{ze^{\beta A}}) \leq z\delta^d(e^{\beta A} - 1)  - c_S + \beta \be_0 \sharp - \ln \left((1-\sharp) + z \delta^d \sharp \right). 
 		\end{align}
 		Therefore, according to \cite{georgii_livre}[Lemma 15.11 and Proposition 15.14] the empirical field $\left(\tilde{\P}_{\Lambda_n}\Ssup{\sharp}\right)_{n\in\N}$ has an accumulation point denoted by $\P\Ssup{\sharp}$ that is stationary by construction. We now prove that this limiting measure satisfies the DLR equations. Let $\Delta \in \Bcal_b(\R^d)$ and let $f$ be a bounded local function. We define a function $f_\Delta$ as 
 		\begin{align*}
 			f_\Delta : \omega \mapsto \int  f(\eta_\Delta \cup \omega_{\Delta^c})  \frac{e^{-\beta H_\Delta(\eta_\Delta \cup \omega_{\Delta^c} )}}{Z_\Delta(\omega_{\Delta^c})} \Pi_\Delta^z(\d \eta)
 		\end{align*}
 		where $Z_\Delta(\omega_{\Delta^c})$ is the normalisation constant given by $Z_\Delta(\omega_{\Delta^c}) = \int e^{-\beta H_\Delta(\eta_\Delta \cup \omega_{\Delta^c})} \Pi_\Delta^z(\d \eta)$. Since $H$ has finite range  \ref{H:finiterange}, then $f_\Delta$ is a bounded local function and, for $n$ large enough 
 		\begin{align*}
 			\int f_\Delta(\omega) \tilde{\P}_{\Lambda_n}\Ssup{\sharp}(\d\omega) &= \frac{1}{\delta^d |\Lambda_n|} \int\limits_{\widehat{\Lambda}_n} \int f_\Delta(\omega + u) \P_{\Lambda_n}\Ssup{\sharp}(\d\omega) \d u \\
 			& = \frac{1}{\delta^d |\Lambda_n|} \int\limits_{\widehat{\Lambda}_n} \iint f(\eta_\Delta \cup (\omega+u)_{\Delta^c}) \frac{e^{-\beta H_\Delta(\eta_\Delta \cup (\omega+u)_{\Delta^c})}}{Z_\Delta((\omega+u)_{\Delta^c})} \Pi_\Delta^z(\d\eta) \P_{\Lambda_n}\Ssup{\sharp}(\d\omega) \d u \\
 			&= \frac{1}{\delta^d |\Lambda_n|} \int\limits_{\widehat{\Lambda}_n} \iint f((\eta_{\Delta-u} \cup \omega_{(\Delta-u)^c})+u) \frac{e^{-\beta H_{\Delta-u}(\eta_{\Delta-u} \cup \omega_{(\Delta-u)^c})}}{Z_{\Delta-u}( \omega_{(\Delta-u)^c})} \Pi_{\Delta-u}^z(\d\eta) \P_{\Lambda_n}\Ssup{\sharp}(\d\omega) \d u.
 		\end{align*}
 		We define the subset $\widehat{\Lambda}_n^* = \{u \in \widehat{\Lambda}_n, \Delta-u \subset \widehat{\Lambda_n \setminus \partial_{int}\Lambda_n} \}$, when $u \in \widehat{\Lambda}_n^*$, we know that $\P_{\Lambda_n}\Ssup{\sharp}$ satisfies DLR equations on $\Delta-u$ by Proposition \ref{prop: finite_volume_DLR} and therefore
 		\begin{align*}
 			\iint f((\eta_{\Delta-u} \cup \omega_{(\Delta-u)^c})+u) \frac{e^{-\beta H_{\Delta-u}(\eta_{\Delta-u} \cup \omega_{(\Delta-u)^c})}}{Z_{\Delta-u}( \omega_{(\Delta-u)^c})} \Pi_{\Delta-u}^z(\d\eta) \P_{\Lambda_n}\Ssup{\sharp}(\d\omega) = \int f(\omega+u) \P_{\Lambda_n}\Ssup{\sharp}(\d\omega).
 		\end{align*}
 		Now we need to deal with the boundary terms, i.e.
 		\begin{gather*}
 			B_1 = \int\limits_{\widehat{\Lambda}_n \setminus \widehat{\Lambda}_n^*} 	\iint f((\eta_{\Delta-u} \cup \omega_{(\Delta-u)^c})+u) \frac{e^{-\beta H_{\Delta-u}(\eta_{\Delta-u} \cup \omega_{(\Delta-u)^c})}}{Z_{\Delta-u}( \omega_{(\Delta-u)^c})} \Pi_{\Delta-u}^z(\d\eta) \P_{\Lambda_n}\Ssup{\sharp}(\d\omega) \d u \\
 			B_2 = \int\limits_{\widehat{\Lambda}_n \setminus \widehat{\Lambda}_n^*} \int f(\omega+u) \P_{\Lambda_n}\Ssup{\sharp}(\d \omega).
 		\end{gather*}
 		Since $f$ is bounded we have that
 		\begin{align*}
 			|B_1| + |B_2| \leq 2M \lambda\left(\widehat{\Lambda}_n \setminus \widehat{\Lambda}_n^*\right).
 		\end{align*}
 		We can observe that $\lambda(\widehat{\Lambda}_n^*)$ is equivalent to $\lambda(\widehat{\Lambda}_n)$ and thus, for some sub-sequence $n_k$, we have
 		\begin{align*}
 			\int f_\Delta(\omega) \P\Ssup{\sharp}(\d\omega) &= \lim\limits_{k \to +\infty} \frac{1}{\delta^d |\Lambda_{n_k}|} \int\limits_{\Lambda_{n_k}} f_\Delta(\omega) \tilde{\P}_{\Lambda_{n-k}}\Ssup{\sharp}(\d\omega) \\
 			&= \lim\limits_{k \to +\infty} \frac{1}{\delta^d |\Lambda_{n_k}|} \int\limits_{\Lambda_{n_k}} \int f(\tau_u(\omega)) \P_{\Lambda_{n_k}}\Ssup{\sharp} (\d \omega) \\
 			&= \lim\limits_{k \to +\infty} \int f(\omega) \tilde{\P}_{\Lambda_{n-k}}\Ssup{\sharp}(\d\omega)\\
 			&= \int f(\omega) \P\Ssup{\sharp}(\d\omega).
 		\end{align*}
 		Therefore, $\P\Ssup{\sharp}$ satisfies the DLR equations and is a Gibbs measure, $\P\Ssup{\sharp} \in \Gcal(H,z,\beta)$. 
 	\end{proof}
	We define the pressure for each boundary condition as 
	\begin{equation}
		\psi\Ssup{\sharp} := \lim\limits_{n \to \infty} \frac{1}{\delta^d |\Lambda_n|} \ln Z_{\Lambda_n}\Ssup{\sharp}.
	\end{equation}
	In the following lemma, we prove that the pressure does not depend on the boundary condition. 
	
	\begin{lemma}\label{lem: pression_inde_bord}
		For any $z>0$, $\beta \geq0$ we have 
		\begin{equation}\label{eq: pression_inde_bord}
			\psi = \psi\Ssup{0} = \psi\Ssup{1}.
		\end{equation}
	\end{lemma}
	
	\begin{proof}
		We denote by $p = \max \left(\left \lceil\frac{L}{\delta} \right \rceil, \left \lceil\frac{r}{\delta} \right \rceil \right)$, $P_n = \Lambda_n \setminus \Lambda_{n-p}$ and $B_n = \Lambda_{n-p} \setminus \Lambda_{n-2p}$. Since $E_0$ satisfies \ref{E:localF}, for any configuration $\omega \subset \widehat{\Lambda_n}$ we have
		\begin{align*}
			H(\omega) = E_{\Lambda_{n+p}}(\omega).
		\end{align*}
		Therefore,
		\begin{align*}
			Z_{\widehat{\Lambda}_n}^{z,\beta} = \int e^{-\beta E_{\Lambda_{n+p}}(\omega)} \Pi_{\widehat{\Lambda}_n}^z(\d\omega).
		\end{align*} 
		We consider the event $F_n$ given by
		\begin{equation*}
			F_n = \bigcap_{i \in P_n \cup B_n} \{ N_{T_i}(\omega) = 1\} \cap \bigcap_{j \in  B_{n-2p} \cup P_{n-2p}} \{N_{T_j}(\omega) \geq 1\}.
		\end{equation*}
		For any $\omega \in F_n$ we have saturation for the tiles in $B_n$ and $P_{n-2p}$ and with it independence between the configurations in the bulk $\Lambda_{n-2p}$ and in the boundary $B_n \cup P_n$. As a consequence, we obtain the following
		\begin{align*}
			Z_{\widehat{\Lambda}_n}^{z,\beta} &\geq \int e^{-\beta (E_{\Lambda_{n+p} \setminus \Lambda_{n-p}} + \overline{E}_{B_n})} e^{-\beta (\overline{E}_{P_{n-2p}} +E_{\Lambda_{n-3p}})} \mathbbm{1}_{F_n}  \Pi_{\widehat{\Lambda}_n}^z(\d\omega) \\
			& \geq Z_{\Lambda_{n-2p}}^{(1)} \int e^{-\beta (E_{\Lambda_{n+p} \setminus \Lambda_{n-p}} + \overline{E}_{B_n})} \mathbbm{1}_{F_n} \Pi_{\widehat{B_n \cup P_n}}^z(\d\omega).
		\end{align*}
		Since $E_0$ satisfies \ref{E:tame} we have  
		\begin{align*}
			E_{\Lambda_{n+p}\setminus \Lambda_{n-p}} \leq c_t \sum_{i \in \Lambda_{n+p}\setminus\Lambda_{n-p}} 1 + N_{T_i \oplus B(0,r)}^2(\omega).
		\end{align*}
		For $\omega \in F_n$ and $i \in \Lambda_{n+p}\setminus\Lambda_{n-p}$ we have 
		\begin{align*}
			N_{T_i \oplus B(0,R)}(\omega) &\leq \#\{j \in \Z^d, T_i \cap (T_j \oplus B(0,r)) \neq \emptyset\} \\
			&\leq \frac{|B(0,r+\sqrt{d}\delta)|}{\delta^d}.
		\end{align*}
		Thus for any $\omega \in F_n$, 
		\begin{align*}
			E_{\Lambda_{n+p}\setminus \Lambda_{n-p}}(\omega) \leq c_t\left(1+ \frac{|B(0,r+\sqrt{d}\delta)|^2}{\delta^{2d}}\right)|\Lambda_{n+p}\setminus\Lambda_{n-p}|.
		\end{align*}
		Furthermore, we have for $\omega \in F_n$ 
		\begin{align*}
			\overline{E}_{B_n}(\omega) = \be_0 |B_n|.
		\end{align*}
		Therefore, there is a constant $c>0$ such that
		\begin{align*}
			\int e^{-\beta (E_{\Lambda_{n+p} \setminus \Lambda_{n-p}} + \overline{E}_{B_n})} \mathbbm{1}_{F_n} \Pi_{\widehat{B_n \cup P_n}}^z(\d\omega) \geq (z \delta^d e^{- (z\delta^d + \beta c)})^{  |\Lambda_{n+p} \setminus \Lambda_{n-2p}|}.
		\end{align*}
		As a consequence, we obtain
		\begin{align*}
			\frac{\ln Z_{\widehat{\Lambda}_n}^{z,\beta}}{\delta^d |\Lambda_n|} \geq \frac{|\Lambda_{n+p}\setminus\Lambda_{n-2p}|}{|\Lambda_n|} \frac{\ln(z\delta^d e^{-(z\delta^d+\beta c)})}{\delta^d} + \frac{|\Lambda_{n-2p}|}{|\Lambda_{n}|} \frac{\ln Z_{\Lambda_{n-2p}}\Ssup{1}}{\delta^d|\Lambda_{n-2p}|}.
		\end{align*}
		By taking the limit we get $\psi \geq \psi^{(1)}$. Now let us consider the event $E_n$ defined as 
		\begin{equation*}
			E_n = \bigcap_{i \in B_n \cup P_n} \{ N_{T_i}(\omega) = 1\} \cap \bigcap_{j \in \Lambda_{n-2p} \setminus \Lambda_{n-5p}} \{N_{T_j}(\omega) =0\}.
		\end{equation*}
		For any configuration $\omega \in E_n$ the tiles in $P_{n-3p}$ are saturated by the empty space and we retrieve the empty boundary condition on $\Lambda_{n-3p}$. Therefore, we have
		\begin{align*}
			Z_{\Lambda_n}^{(1)} &\geq \int e^{-\beta(\overline{E}_{P_{n}}+E_{\Lambda_{n-p}\setminus\Lambda_{n-3p}})} e^{-\beta ( \overline{E}_{P_{n-3p}} + E_{\Lambda_{n-4p}})}\mathbbm{1}_{E_n} \Pi_{\widehat{\Lambda}_n}^z(\d\omega) \\
			& \geq Z_{\Lambda_{n-3p}}^{(0)} \int e^{-\beta(\overline{E}_{P_{n}}+E_{\Lambda_{n-p}\setminus\Lambda_{n-3p}})} \mathbbm{1}_{E_n}(\omega) \Pi_{\Lambda_n \setminus \Lambda_{n-3p}}^z(\d\omega).
		\end{align*}
		With similar arguments we can show that there exists $c>0$ such that 
		\begin{align*}
			\int e^{-\beta(\overline{E}_{P_{n}}+E_{\Lambda_{n-p}\setminus\Lambda_{n-3p}})} \mathbbm{1}_{E_n} \Pi_{\Lambda_n \setminus \Lambda_{n-3p}}^z(\d\omega) \geq (z\delta^d e^{-(z\delta^d + \beta c)})^{|\Lambda_n \setminus \Lambda_{n-3p}|}.
		\end{align*}
		Therefore, we have $\psi^{(1)}\geq \psi^{(0)}$. Finally, let us recall that $(0)_{\Lambda_n} = \bigcap_{i \in P_n \cup B_n} \{ N_{T_i}(\omega) =0\}$ and therefore
		\begin{align*}
			Z_{\Lambda_n}^{(0)} &= \int e^{-\beta E_{\Lambda_{n-p}} } \mathbbm{1}_{(0)_{\Lambda_n}} \Pi_{\widehat{\Lambda}_n}^z(\d\omega) \\
			&= e^{-z \delta^d |P_n \cup B_n|} \int e^{-\beta H }  \Pi_{\widehat{\Lambda}_{n-2p}}^z(\d\omega) \\
			&= e^{-z \delta^d |P_n \cup B_n|} Z_{\widehat{\Lambda}_{n-p}}^{z,\beta}.
		\end{align*}
		As a result, we have $\psi^{(0)} = \psi$  which finishes the proof.
	\end{proof}
	
	\subsection{Polymer development}
	We have proved that a Gibbs measure can be obtained as limits of empirical fields for each boundary condition. We want to prove that these measures are different. In order to do that, we will compare the partition function of the model with boundary conditions, $Z_\Lambda\Ssup{\sharp}$, with the partition function of the model that saturates each tile, $\overline{Z}_\sharp^{|\Lambda|}$, where
	\begin{equation}
		\overline{Z}_\sharp = \int e^{-\beta \overline{E}_0(\omega)} \1_{\{\sigma(\omega,0)=\sharp\}} \Pi_{T_0}^z(\d \omega) =\begin{cases}
			e^{-z\delta^d} & \text{if  } \sharp = 0 \\
			e^{-\beta \be_0}(1- e^{-z \delta^d}) & \text{if  } \sharp = 1.
		\end{cases}
	\end{equation}
	For this purpose, we analyse the ratio of partition functions for any subset $\Lambda \subset \Z^d$ given by
	\begin{equation}
		\Phi_\Lambda\Ssup{\sharp} := \frac{Z_\Lambda\Ssup{\sharp}}{\overline{Z}_\sharp^{|\Lambda|}}.
	\end{equation}
	We define a new boundary for any $\Lambda \subset \Z^d$ 
	\begin{equation*}
		\partial^{\ssup{-}} \Lambda := \{ i \in \Lambda, \delta d_2(i,\Lambda^c) \leq L\}. 
	\end{equation*}
	For any configuration $\omega \in \Omega$ and any contour $\gamma(\omega)$, the tiles in $\partial^{\ssup{-}} \overline{\gamma}(\omega) \cup \partial \overline{\gamma}(\omega) $ are saturated. In the following proposition, we prove that this ratio can be written as a polymer development.
	\begin{proposition}\label{prop: polymer_dvt}
		For any $\Lambda \subset \Z^d$ finite and any $\sharp \in \{0,1\}$, we have 
		\begin{equation*}
			\Phi_\Lambda\Ssup{\sharp} = \sum_{\Gamma \in  \mathcal{C}\Ssup{\sharp}(\Lambda)}  \prod_{\gamma \in \Gamma} w_{\gamma}\Ssup{\sharp},
		\end{equation*}
		where $w_{\gamma}\Ssup{\sharp} = \overline{Z}_\sharp^{-|\overline{\gamma}|} I_{\gamma} \frac{Z_{\Int_{\sharp^*} \gamma}\Ssup{\sharp^*}}{Z_{\Int_{\sharp^*} \gamma}\Ssup{\sharp}}$, $\sharp^* := 1 - \sharp$ and  
		\begin{equation*}
			I_{\gamma} := \int e^{-\beta \left(E_{\overline{\gamma} \setminus \partial^-\overline{\gamma}}  + \overline{E}_{\partial^-\overline{\gamma}}\right)(\omega)} \1_{\{\forall i \in \overline{\gamma}, \sigma(\omega,i)=\sigma_i \}} \Pi_{\widehat{\gamma}}^z(\d\omega).
		\end{equation*}
		 The quantity, $w_{\gamma}\Ssup{\sharp}$, is called the weight of the contour $\gamma$.
	\end{proposition}
	The proof of this proposition follows the same outline as in \cite[Proposition 2]{renaudchanQuermass}. Therefore, the analysis of $\Phi_\Lambda\Ssup{\sharp}$ can be done via cluster expansion methods. For that, we need to show that for $\tau >0$ large enough the weights satisfy
	\begin{gather}
		w_\gamma\Ssup{\sharp} \leq e^{-\tau |\overline{\gamma}|}, \label{tau stab}\\
		\left| \frac{\partial w_\gamma\Ssup{\sharp}}{\partial z}\right| \leq |\overline{\gamma}|^{\frac{d}{d-1}} e^{-\tau |\overline{\gamma}|}`. \label{tau stab derivative}
	\end{gather}
	A weight is said to be $\tau$-stable if it satisfies \eqref{tau stab}. In the following lemma, we obtain nice bounds on $I_\gamma$ if the interaction satisfies a Peierls condition.
	\begin{lemma}\label{lem: bound_I_gamma}
		We assume that for any contours $\gamma$ and any configuration $\omega$ that achieve the contour $\gamma$ there exists $\be_+> 0$ such that 
		\begin{equation} \label{ineq: peierls_cond}
			E_{\overline{\gamma} \setminus \partial^{\ssup{-}} \overline{\gamma}} - \overline{E}_{\overline{\gamma} \setminus \partial^{\ssup{-}} \overline{\gamma}}(\omega) \geq \be_+ |\overline{\gamma}|.
		\end{equation} 
		Then we have 
		\begin{gather}
			I_\gamma \leq \overline{Z}_0^{|\overline{\gamma}_0|}\overline{Z}_1^{|\overline{\gamma}_1|} e^{-\beta \be_+ |\overline{\gamma}|} \label{ineq: peierls_cond_1}\\ 
			\left|\frac{\partial I_\gamma}{\partial z}\right| \leq \frac{2- e^{-z\delta^d}}{1 - e^{-z \delta^d}}|\overline{\gamma}|\delta^d \overline{Z}_0^{|\overline{\gamma}_0|}\overline{Z}_1^{|\overline{\gamma}_1|} e^{-\beta \be_+ |\overline{\gamma}|} \label{ineq: peierls_cond_1_derivative}
		\end{gather}
		where $\overline{\gamma}_\sharp := \{ i \in \overline{\gamma}, \sharp_i = \sharp \}$.
	\end{lemma}
	
	\begin{proof}
		For the proof of \eqref{ineq: peierls_cond_1}, we simply inject \eqref{ineq: peierls_cond} into $I_\gamma$ and have
		\begin{align*}
			I_\gamma &\leq e^{-\beta \be_+ |\overline{\gamma}|} \int e^{-\beta \overline{E}_{\overline{\gamma}}(\omega)} \1_{\{\forall i \in \overline{\gamma}, \sigma(\omega,i)=\sigma_i \}} \Pi_{\widehat{\gamma}}^z(\d\omega) \\
			& \leq \overline{Z}_0^{|\overline{\gamma}_0|} \overline{Z}_1^{|\overline{\gamma}_1|} e^{-\beta \be_+ |\overline{\gamma}|}.
		\end{align*}
		For the second inequality \eqref{ineq: peierls_cond_1_derivative}, we need to observe that 
		\begin{align*}
			\frac{\partial I_\gamma}{\partial z} = - \delta^d |\overline{\gamma}| I_\gamma + \frac{1}{z} \int N_{\widehat{\overline{\gamma}}}(\omega) e^{-\beta \left(E_{\overline{\gamma} \setminus \partial^-\overline{\gamma}}  + \overline{E}_{\partial^-\overline{\gamma}}\right)(\omega)} \1_{\{\forall i \in \overline{\gamma}, \sigma(\omega,i)=\sigma_i \}} \Pi_{\widehat{\gamma}}^z(\d\omega).
		\end{align*}
		By injecting \eqref{ineq: peierls_cond} into the second term we have
		\begin{align*}
			\int N_{\widehat{\overline{\gamma}}}(\omega) e^{-\beta \left(E_{\overline{\gamma} \setminus \partial^-\overline{\gamma}}  + \overline{E}_{\partial^-\overline{\gamma}}\right)(\omega)} &\1_{\{\forall i \in \overline{\gamma}, \sigma(\omega,i)=\sigma_i \}} \Pi_{\widehat{\gamma}}^z(\d\omega) \\ &\leq e^{-\beta \be_+ |\overline{\gamma}|} \int N_{\widehat{\overline{\gamma}}}(\omega) e^{-\beta \overline{E}_{\overline{\gamma}}(\omega)} \1_{\{\forall i \in \overline{\gamma}, \sigma(\omega,i)=\sigma_i \}} \Pi_{\widehat{\gamma}}^z(\d\omega) \\
			& \leq e^{-\beta \be_+ |\overline{\gamma}|} \sum_{i\in \overline{\gamma}_1}\int N_{T_i}(\omega) e^{-\beta \overline{E}_{\overline{\gamma}}(\omega)} \1_{\{\forall i \in \overline{\gamma}, \sigma(\omega,i)=\sigma_i \}} \Pi_{\widehat{\gamma}}^z(\d\omega) \\
			&\leq e^{-\beta \be_+ |\overline{\gamma}|} \overline{Z}_0^{|\overline{\gamma}_0|} \overline{Z}_1^{|\overline{\gamma}_1|} \sum_{i\in \overline{\gamma}_1} \frac{1}{\overline{Z}_1} \int N_{T_i}(\omega) e^{-\beta \be_0} \1_{\{\forall i \in \overline{\gamma}, \sigma(\omega,i)=1 \}} \Pi_{T_i}^z(\d\omega) \\
			& \leq \frac{z\delta^d}{1-e^{-z\delta^d}} |\overline{\gamma}_1|  \overline{Z}_0^{|\overline{\gamma}_0|} \overline{Z}_1^{|\overline{\gamma}_1|} e^{-\beta \be_+ |\overline{\gamma}|}.
		\end{align*}
		Using this previous upper bound and \eqref{ineq: peierls_cond_1}, we obtain \eqref{ineq: peierls_cond_1_derivative}.
	\end{proof}
	Under the assumption of a Peierls bound and according to Lemma \ref{lem: bound_I_gamma}, the weights of the contours and their derivative with respect to $z$ verify, for $z \in O_\beta$, 
	\begin{gather}
		w_\gamma\Ssup{\sharp} \leq e^{-(\beta \be_0 - 2)|\overline{\gamma}|} \frac{Z_{\Int_{\sharp^*} \gamma}\Ssup{\sharp^*}}{Z_{\Int_{\sharp^*} \gamma}\Ssup{\sharp}}, \label{ineq: quasi_tau_stab}  \\
		\left| \frac{\partial w_\gamma\Ssup{\sharp}}{\partial z} \right| \leq    \left(|\overline{\gamma}| \left( \delta^d \frac{2- e^{-z\delta^d}}{1 - e^{-z \delta^d}} + \left|\displaystyle\frac{\frac{\partial Z_\sharp}{\partial z}}{Z_\sharp} \right| \right) \frac{Z_{\Int_{\sharp^*} \gamma}\Ssup{\sharp^*}}{Z_{\Int_{\sharp^*} \gamma}\Ssup{\sharp}} + \frac{\partial}{\partial z} \left( \frac{Z_{\Int_{\sharp^*} \gamma}\Ssup{\sharp^*}}{Z_{\Int_{\sharp^*} \gamma}\Ssup{\sharp}} \right) \right) e^{-(\beta \be_0 -2) |\overline{\gamma}| }. \label{ineq: quasi_tau_stab_deriv}
	\end{gather} 
	At this point of the proof, we need to deal with the contributions of the ratio of partition functions in order to obtain the desired bounds for the weights, \eqref{tau stab} and \eqref{tau stab derivative}.  
	\subsection{Truncated pressures and weights}
	In general, the ratio of partition functions for different boundary conditions might grow too fast with the volume of $\Lambda$ and therefore the weights of some contours might not satisfy \eqref{tau stab} and \eqref{tau stab derivative}. In order to deal with this, we will study the truncated version of the weights. 
	\newline
	
	Let us consider a cut-off function $\kappa \in \Ccal^1(\R, [0,1])$ such that 
	\begin{gather}
		\kappa(s) = \begin{cases}
			1, &\text{ if } s\leq \frac{\beta \be_+}{8}, \\
			0, &\text{ if } s\geq \frac{\beta \be_+}{4}, 
		\end{cases} \\
		\ \norm{\kappa'} = \sup_\R |\kappa'(s)| <\infty. 
	\end{gather}
	The construction of the truncated weights and pressure is inductive with respect to the size of the interior of the contours. For $n=0$, we define the truncated pressure as 
	\begin{equation}
		\widehat \psi_0\Ssup{\sharp} := \frac{\ln Z_\sharp}{\delta^d},
	\end{equation}
	and for any contour $\gamma \in \Ccal_0\Ssup{\sharp}$,
	\begin{equation}
		\widehat{w}_\gamma\Ssup{\sharp} := w_\gamma\Ssup{\sharp} = \overline{Z}_\sharp^{-|\overline{\gamma}|} I_\gamma.  
	\end{equation} 
	Now we suppose that the truncated weights are well defined for contours $\gamma$ such that  $|\Int \gamma| \leq n$. We define the truncated partition function at rank $n$ to be
	\begin{equation}
		\widehat{Z}_{\Lambda, n}\Ssup{\sharp} := Z_\sharp^{|\overline{\Lambda}|}\sum_{\Gamma \in \Ccal_n\Ssup{\sharp}(\Lambda)} \prod_{\gamma \in \Gamma} \widehat{w}_\gamma\Ssup{\sharp},
	\end{equation}
	and the associated truncated pressure for boundary $\sharp$ at rank $n$ is given by
	\begin{equation}
		\widehat{\psi}_n\Ssup{\sharp} := \lim\limits_{k \to \infty} \frac{\ln \widehat{Z}_{\Lambda_k, n}}{\delta^d |\Lambda_k|}.
	\end{equation}
	Using a sub-additivity argument we prove that this limit exists. By construction, we have that $\overline{Z}_\sharp^{|\Lambda|} \leq \widehat Z_{\Lambda, n} \leq Z_\Lambda$ since we have at least the empty contour $\gamma = (\emptyset, \emptyset)$ whose weight is by definition equal to $1$. Therefore, $\widehat{\psi}_n\Ssup{\sharp} \in [\widehat{\psi}_0\Ssup{\sharp}, \psi]$. We can observe that $\left(\widehat{\psi}_n\Ssup{\sharp}\right)_{n \in \N}$ is increasing. Furthermore, we define the truncated pressure as 
	\begin{equation}
		\widehat{\psi}_n := \max \{\widehat{\psi}\Ssup{0}, \widehat{\psi}\Ssup{1}\}. 
	\end{equation}
	Let us consider a contour $\gamma$ such that $|\Int \gamma| = n+1$. We define the truncated weight of $\gamma$ as 
	\begin{equation}
		\widehat{w}_{\gamma}\Ssup{\sharp} := \overline{Z}_\sharp^{-|\overline{\gamma}|} I_{\gamma} \, \kappa \left( (\widehat{\psi}_n\Ssup{\sharp^*} - \widehat{\psi}_n\Ssup{\sharp}) \delta^d |\Int_{\sharp^*} \gamma|^{\frac{1}{d}}  \right) \frac{Z_{\Int_{\sharp^*} \gamma}\Ssup{\sharp^*}}{Z_{\Int_{\sharp^*} \gamma}\Ssup{\sharp}}.
	\end{equation}
	Furthermore, we define the following quantities of interest, $a_n\Ssup{\sharp} := \widehat{\psi}_n - \widehat{\psi}_n\Ssup{\sharp}$. By definition of the truncated weights for all contours $\gamma$ of class $n+1$ we have the following implication
	\begin{equation}
		a_n\Ssup{\sharp} \delta^d (n+1)^{\frac{1}{d}} \leq \frac{\beta \be_+}{8} \implies \widehat{w}_{\gamma}\Ssup{\sharp}  = w_{\gamma}\Ssup{\sharp}.
	\end{equation}
	Once the construction of the truncated weights for all contours is completed, we can define the truncated partition function associated with the $\sharp$ boundary as
	\begin{equation}
		\widehat{Z}_\Lambda\Ssup{\sharp} := \overline{Z}_\sharp^{|\Lambda|}  \widehat{\Phi}_\Lambda\Ssup{\sharp}, \quad \text{ where } \quad  \widehat{\Phi}_\Lambda\Ssup{\sharp}= \sum_{\Gamma \in \Ccal\Ssup{\sharp}(\Lambda)} \prod_{\gamma \in \Gamma} \widehat{w}_\gamma\Ssup{\sharp},
	\end{equation}
	the truncated pressure for the $\sharp$ boundary as
	\begin{equation}
		\widehat{\psi}\Ssup{\sharp} := \lim\limits_{k \to \infty} \frac{\ln \widehat{Z}_{\Lambda_k}\Ssup{\sharp}}{\delta^d |\Lambda_k|},
	\end{equation}
	and the truncated pressure as
	\begin{equation}
		\widehat{\psi} := \max\{\widehat{\psi}\Ssup{0}, \widehat{\psi}\Ssup{1}\}.
	\end{equation} 
	Equivalently, the truncated pressures can also be obtained as the limit of the truncated pressures of rank $n$. We denote $a\Ssup{\sharp} := \widehat{\psi} - \widehat{\psi}\Ssup{\sharp}$. Note that the truncated pressure for the $\sharp$ boundary satisfies
	\begin{equation}
		\widehat{\psi}\Ssup{\sharp} = \widehat{\psi}_0\Ssup{\sharp} + \varepsilon\Ssup{\sharp},
	\end{equation}
	where
	\begin{equation}
		\varepsilon\Ssup{\sharp} = \lim\limits_{k \to \infty} \frac{\ln \widehat{\Phi}_{\Lambda_k}\Ssup{\sharp}}{\delta^d |\Lambda_k|}. 
	\end{equation}
	
	\begin{proposition}\label{prop: tau_stab_trunc_weights}
		Let $\tau := \frac{1}{2}\beta  \be_+ - 8$. Then there exists $\beta_0$ with  $0<\beta_0< \infty$ such that, for all $\beta > \beta_0$, there exist $C_1>0$, $C_2>0$ and $D\geq 1$ for which the following statements hold for all $z \in O_\beta$, $\sharp \in \{0,1\}$ and $n\geq 0$.
		\begin{enumerate}
			\item For all $k \leq n$, the truncated weights of each contour $\gamma$ with $|\Int \gamma| = k$ verify
			\begin{equation}
				\widehat{w}_{\gamma}\Ssup{\sharp} \leq e^{-\tau |\overline{\gamma}|} \label{truncated_weights_tau_stability}
			\end{equation}
			and we have the following implication
			\begin{equation} \label{implication_bound_a_n_to_stability}
				a_n\Ssup{\sharp} \delta^d |\Int \gamma|^{\frac{1}{d}} \leq \frac{\beta \be_+ }{16} \implies \widehat{w}_{\gamma}\Ssup{\sharp} = w_{\gamma}\Ssup{\sharp}. 
			\end{equation}
			Moreover, $z \mapsto \widehat{w}_{\gamma}\Ssup{\sharp} $ is differentiable and satisfies
			\begin{equation}
				\left|\frac{\partial \widehat{w}_{\gamma}\Ssup{\sharp}}{\partial z} \right| \leq D |\overline{\gamma}|^{\frac{d}{d-1}} e^{-\tau |\overline{\gamma}|}. \label{bound_truncated_weight_derivative}
			\end{equation}
			\item For $\Lambda \subset \Z^d$ such that $|\Lambda| \leq n+1$, we have
			\begin{align}
				Z_\Lambda\Ssup{\sharp}& \leq e^{\widehat{\psi}_{n} \delta^d|\Lambda| + 2 |\partial_{ext} \Lambda|}, \label{bound_partition_function}\\
				\left| \frac{\partial Z_\Lambda\Ssup{\sharp}}{\partial z} \right| & \leq \left( C_1 |\Lambda| + C_2 |\partial_{ext} \Lambda| \right) e^{ \widehat{\psi}_n \delta^d|\Lambda| + 2 |\partial_{ext} \Lambda|}. \label{bound_derivative_partition_function}
			\end{align}
		\end{enumerate}
	\end{proposition}
	
	The proof of this proposition follows the same line of argument as that of \cite[Proposition 4]{renaudchanQuermass}, with 
	\begin{gather}
		C_1 = \sup_{z\in O_\beta} \left\{ (e+2)\delta^d + \frac{\beta (\be_0 +\min(\be_0,0) ) - \ln(1-e^{-z\delta^d})}{z} \right\},  \quad  C_2 = \frac{1}{z_\beta^-},		
	\end{gather}
	and
	\begin{equation}
		D = 3 C_1   + 2 C_2 + \left( \frac{1}{1-e^{-z_\beta^- \delta^d}} + 2 \right) \norm{\kappa'} \delta^d
	\end{equation}
	
	As a consequence of Proposition \ref{prop: tau_stab_trunc_weights}, we can conclude that all the truncated weights are $\tau$-stable and, for $\beta$ large enough the truncated pressure $\widehat{\psi}\Ssup{\sharp}$ can be expressed via a convergent cluster expansion. Furthermore, since the truncated weights are $\tau$-stable by \cite[Lemma 12]{renaudchanQuermass}, for $n\geq k$, we have 
	\begin{equation}
		|\widehat{\psi}_n\Ssup{\sharp} - \widehat{\psi}_k\Ssup{\sharp}| \leq \frac{1}{ \delta^d} e^{-\frac{\tau}{2}k^{\frac{d-1}{d}}} \quad \text{and } \quad |\widehat{\psi}_{n} - \widehat{\psi}_{k}| \leq \frac{1}{ \delta^d} e^{-\frac{\tau}{2}k^{\frac{d-1}{d}}}.
	\end{equation}
	As $n$ goes to infinity, we obtain
	\begin{equation}
		|\widehat{\psi}\Ssup{\sharp} - \widehat{\psi}_k\Ssup{\sharp}| \leq \frac{1}{ \delta^d} e^{-\frac{\tau}{2}k^{\frac{d-1}{d}}} \quad \text{and } \quad |\widehat{\psi} - \widehat{\psi}_{k}| \leq \frac{1}{ \delta^d} e^{-\frac{\tau}{2}k^{\frac{d-1}{d}}}.
	\end{equation}
	Therefore, for any contours $\gamma$ of class $k$, we have
	\begin{align*}
		a_k\Ssup{\sharp} \delta^d | \Int \gamma|^{\frac{1}{d}} &= a\Ssup{\sharp} \delta^d |\Int \gamma|^{\frac{1}{d}} +(a_k\Ssup{\sharp} - a\Ssup{\sharp})\delta^d |\Int \gamma|^{\frac{1}{d}} \\
		& \leq a\Ssup{\sharp} \delta^d |\Int \gamma|^{\frac{1}{d}} + 2 k^{\frac{1}{d}} e^{-\frac{\tau k^{\frac{1}{d}}}{2}}.
	\end{align*}
	Thus, for $\beta$ large enough such that $2 k^{\frac{1}{d}} e^{-\frac{\tau k^{\frac{1}{d}}}{2}} \leq \frac{\beta \be_+}{16}$,
	we get
	\begin{align*}
		a_k\Ssup{\sharp} \delta^d | \Int \gamma|^{\frac{1}{d}} \leq a\Ssup{\sharp} \delta^d |\Int \gamma|^{\frac{1}{d}} + \frac{\beta \be_+}{16}.
	\end{align*}
	As a consequence, the implication \eqref{implication_bound_a_n_to_stability} still holds as $n$ goes to infinity. Therefore, when $a\Ssup{\sharp} =0$, for any contour $\gamma$, we have $\widehat{w}_\gamma\Ssup{\sharp} = w_\gamma\Ssup{\sharp}$ and, for any $\Lambda \subset \Z^d$, $\widehat{Z}_\Lambda\Ssup{\sharp} = Z_\Lambda\Ssup{\sharp}$. This implies that if $\widehat{\psi}\Ssup{\sharp} = \widehat{\psi}$, then $\widehat{\psi}\Ssup{\sharp} = \psi\Ssup{\sharp}$ and, by Lemma \ref{lem: pression_inde_bord}, we have 
	\begin{equation}\label{eq: pression_pressiontronc}
		\psi = \widehat{\psi} = \max\{\widehat{\psi}\Ssup{0}, \widehat{\psi}\Ssup{1} \}. 
	\end{equation}
	
	\subsection{Proof of Theorem \ref{thm: liquid_gaz1}}
	In the previous sub-section, using Proposition \ref{prop: tau_stab_trunc_weights}, we showed that all the truncated weights satisfy \eqref{tau stab} and \eqref{tau stab derivative}. If we assume that $\beta$ is large enough, the truncated weights satisfy the assumptions of \cite[Theorem 11]{renaudchanQuermass} and thus, we have
	\begin{align*}
		\widehat{\Phi}_\Lambda\Ssup{\sharp} = e^{ \varepsilon\Ssup{\sharp} \delta^d  |\Lambda| + \Delta_\Lambda\Ssup{\sharp}}	
	\end{align*}
	and $\varepsilon\Ssup{\sharp}$ and $\Delta_\Lambda\Ssup{\sharp}$ are $\mathcal{C}^1$ in $O_\beta$. Furthermore, according to Theorem \cite[Theorem 11]{renaudchanQuermass} uniformly in $O_\beta$, we have
	\begin{align}
		&|\varepsilon\Ssup{\sharp}| \leq  \eta(\tau, l_0),&  &|\Delta_\Lambda\Ssup{\sharp}| \leq \eta(\tau, l_0) |\partial_{ext} \Lambda |, \label{ineq: quant_pertubatif}\\
		&\left| \frac{\partial \varepsilon\Ssup{\sharp}}{\partial z} \right| \leq D \eta(\tau, l_0),&  &\left| \frac{\partial \Delta_\Lambda\Ssup{\sharp}}{\partial z} \right| \leq D \eta(\tau, l_0) |\partial_{ext} \Lambda|, \label{ineq: quant_perturbatif_derivee}
	\end{align}
	where $\eta(\tau,l_0)= 2e^{-\frac{\tau l_0}{3}}$ and $l_0$ is the size of the smallest contour. Now we look at the difference between the two truncated pressures with boundary conditions. 
	\begin{equation}
		G(z) := \widehat{\psi}\Ssup{1} - \widehat{\psi}\Ssup{0} = - \frac{\beta \be_0}{\delta^d} + \frac{\ln(e^{z \delta^d} -1 )}{\delta^d} + \varepsilon\Ssup{1} - \varepsilon\Ssup{0}.
	\end{equation}
	According to \eqref{ineq: quant_pertubatif}, we have for $\beta$ large enough, 
	\begin{gather}
		G(z_\beta^-) \leq -\frac{2}{\delta^d} + 2 \eta(\tau, l_0) <0, \\
		G(z_\beta^+) \geq \frac{2}{\delta^d} - 2 \eta(\tau, l_0) >0. 
	\end{gather}
	In addition, using \eqref{ineq: quant_perturbatif_derivee}, we have, for $z \in O_\beta$,
	\begin{equation}
		G'(z) = \frac{e^{z \delta^d}}{e^{z \delta^d} - 1} + \frac{\partial \varepsilon\Ssup{1}}{\partial z} - \frac{\partial \varepsilon\Ssup{0}}{\partial z} > 1 - 2 D \eta(\tau, l_0) > 0. 
	\end{equation}
	Therefore, there is a unique $z_\beta^c \in O_\beta$ such that 
	\begin{equation}
		\widehat{\psi} = \begin{cases}
			\widehat{\psi}\Ssup{0} & \text{if } z \in (z_\beta^-, z_\beta^c], \\
			\widehat{\psi}\Ssup{1} & \text{if } z \in [z_\beta^c, z_\beta^+). 
		\end{cases}
	\end{equation}
	Furthermore, we have for all $z \in O_\beta$ 
	\begin{equation}\label{ineq: derivee_pression}
		\frac{\partial \widehat{\psi}\Ssup{1}}{\partial z} > \frac{\partial \widehat{\psi}\Ssup{0}}{\partial z},
	\end{equation}
	and therefore, with \eqref{eq: pression_pressiontronc} we have 
	\begin{equation}
		\frac{\partial \psi}{\partial z^+} (z_\beta^c) > \frac{\partial \psi}{\partial z^-} (z_\beta^c).
	\end{equation}	
	We have found the critical activity for which the pressure is non-differentiable. Our goal now is to prove that the infinite Gibbs measures obtained in Proposition \ref{prop: champ_empirique_DLR} give different densities of particles. By direct computation, we have
	\begin{equation}
		\frac{\partial Z_{\Lambda_k}\Ssup{\sharp}}{\partial z} = - \delta^d |\Lambda_k| + \frac{1}{z} E_{\P_{\Lambda_k}\Ssup{\sharp}}\left( N_{\widehat \Lambda_k} \right)
	\end{equation}
	For $\beta$ large enough and $z = z_\beta^c$, we know that $Z_\Lambda\Ssup{\sharp} = \widehat Z_\Lambda\Ssup{\sharp}$, since $a\Ssup{0} = a\Ssup{1} = 0$. Therefore
	\begin{equation}
		\frac{\partial \ln Z_{\Lambda_k}\Ssup{\sharp}}{\partial z} = \frac{\partial \ln \widehat{Z}_{\Lambda_k}\Ssup{\sharp}}{\partial z} = \frac{\partial \widehat{\psi}\Ssup{\sharp}}{\partial z} \delta^d |\Lambda_k| + \frac{\partial \Delta\Ssup{\sharp}_{\Lambda_k}}{\partial z}.
	\end{equation}
	As a consequence, we obtain the following 
	\begin{equation}
		\frac{E_{\P_{\Lambda_k}\Ssup{\sharp}}\left(N_{\widehat{\Lambda}_k}\right)}{\delta^d |\Lambda_k|} = z + z \frac{\partial \widehat{\psi}\Ssup{\sharp}}{\partial z} + \frac{z}{\delta^d |\Lambda_k|} \frac{\partial \Delta\Ssup{\sharp}_{\Lambda_k}}{\partial z}.
	\end{equation}
	Using \eqref{ineq: quant_perturbatif_derivee}, we have
	\begin{equation}\label{lim: quant pertubatif derivee}
		\lim\limits_{k \to \infty} \frac{1}{|\Lambda_k|} \frac{\partial \Delta_{\Lambda_k}\Ssup{\sharp}}{\partial z} = 0
	\end{equation}
	 By definition of the empirical field, we have 
	 \begin{equation}
	 	E_{\P_{\Lambda_k}\Ssup{\sharp}}\left(N_{\widehat{\Lambda}_k}\right) = E_{\tilde \P_{\Lambda_k}\Ssup{\sharp}}\left(N_{\widehat{\Lambda}_k}\right),
	 \end{equation}
	 and, since it is stationary, we have
	 \begin{equation}
	 	\frac{E_{\tilde \P_{\Lambda_k}\Ssup{\sharp}}\left(N_{\widehat{\Lambda}_k}\right)}{\delta^d |\Lambda_k|} = E_{\tilde \P_{\Lambda_k}\Ssup{\sharp}}\left(N_{[0,1]^d}\right) = \rho(\tilde \P_{\Lambda_k}\Ssup{\sharp}).
	 \end{equation}
	 According to Proposition \ref{prop: champ_empirique_DLR}, we know that the empirical field $\left(\tilde \P_{\Lambda_k}\Ssup{\sharp} \right)_{k \in \N}$ has an accumulation point $\P\Ssup{\sharp}$ for the local topology. Therefore, by taking the limit for the correct sub-sequence we have
	 \begin{equation}
	 	\rho(\P\Ssup{\sharp}) = z_\beta^c + z_\beta^c \frac{\partial \widehat{\psi}\Ssup{\sharp}}{\partial z}(z_\beta^c).
	 \end{equation}
	 Finally, according to \eqref{ineq: derivee_pression} we have $\rho(\P^{(1)}) > \rho(\P^{(0)})$.
	 
	\section{Proof of Theorem \ref{thm:fopt_diluted_pair} }\label{sec: proof_thm2}
	
	We have shown in the section with examples of saturated interactions that the diluted pairwise interaction satisfies assumptions \ref{E:non-degenerate}-\ref{E:saturation}. In order to apply Theorem \ref{thm: liquid_gaz1} we need to show that, under the condition \eqref{cond: phi_global}, the diluted pairwise interaction exhibit a Peierls-like condition. 
	
	\begin{proposition}
		For $\delta$ small enough, $L> R_2 + 2 \sqrt{d}\delta$ and $\phi$ be a radial pair potential such that $r^{d-1}\phi \in L^1$ and that satisfies
		\begin{equation}
			C_d \int\limits_{B(0,R)} \phi^+ \d x > \left[ \left(\frac{R_1}{R} \right)^d-1 \right] \int\limits_{B(0,R_1) \backslash B(0,R)} \phi^+ \d x + \int\limits_{\R^d} \phi^- \d x,
		\end{equation}
		where
		\begin{equation*}\label{cond: diluted_pairwise interaction}
			C_d = \frac{\int_{0}^{\frac{\pi}{3}} \sin(\theta)^{d-2} \d \theta }{\int_{0}^{\pi} \sin(\theta)^{d-2} \d \theta }.
		\end{equation*}
		There is $\be_+>0$ such that for any contours $\gamma$ and any configuration $\omega$ that achieves this contour we have
		\begin{equation}
			E_{\overline{\gamma}} - \overline{E}_{\overline{\gamma}}(\omega) \geq \be_+ |\overline{\gamma}|.
		\end{equation}
	\end{proposition}
	
	For any $x,z \in \R^d$ with $x\neq z$ we denote $B_{sec}(z,x)$ the hyperspherical sector with polar angle $\frac{\pi}{3}$, radius $R$, $z$ as the centre and directed toward $x$.
	\begin{figure}[H]
		\centering
		\begin{tikzpicture}[scale=2]
			\draw (0,0) circle (1.1);
			\draw (0,0) ellipse (1.1 and 0.3);
			\draw[dashed] (30:0.6) arc (30:90:0.6);
			\draw[->] (0,0) -- (0,1.2) node[above] {$x$};
			\fill[blue, opacity=0.3] (0,0) -- (30:1.1) arc (30:150:1.1) -- cycle;
			\node[below] at (0,0) {$z$};
			\node[below] at (-0.4,0.8) {$B_{sec}(z,x)$};
			\node[below] at (0.4, 0.85) {$\frac{\pi}{3}$};
		\end{tikzpicture}
		\caption{$B_{sec}(z,x)$}
	\end{figure}
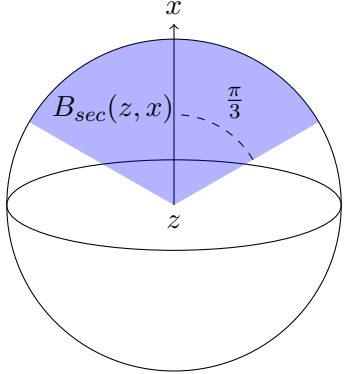
	The constant $C_d$ corresponds to the ratio of the volume of $B_{sec}(z,x)$ and the volume of $B(z,R)$.
	
	\begin{proof}
		First we will consider the case where $R \geq R_1$ and prove that the interaction satisfies a Peierls-like condition using the dominoes \eqref{def: dominoes}. We fix $\delta \leq \frac{R-R_1}{2\sqrt{d}}$. Let's consider a configuration $\omega$ such that $N_{T_0}(\omega) \geq 1$ we have for any $y \in T_0$, $B(y, R_1) \subset L_R(\omega)$ and therefore 
		\begin{equation}
			E_0(\omega) =  \int_{T_0} \int_{L_R(\omega)} \phi(|x-y|)\d x \d y \geq \delta^d C_\phi = \overline{E}_0(\omega).
		\end{equation}
			Next we consider a configuration $\omega$ such that $N_{T_0}(\omega) =0$. If $L_R(\omega) \cap T_0 = \emptyset$, then $E_0(\omega) = \overline{E}_0(\omega) =0$. Otherwise, for any $x \in L_R(\omega) \cap T_0$ and $z \in \omega$, the closest point in the configuration to $x$,  we know that at least we have $B_{sec}(z,x)$ included in the halo. Therefore, we directly obtain the following lower bound 
		\begin{align*}
			\int_{L_R(\omega)-x} \phi \d y \geq \int_{B_{sec}(z,x)} \phi^+\d y - \int_{\R^d} \phi^-\d y.
		\end{align*}
		Since the potential is radial and $R\geq R_1$ we have that 
		\begin{align*}
			\int_{B_{sec}(x,z)} \phi^+ \d y = C_d \int_{B(0,R_1)} \phi^+ \d y.
		\end{align*}
		If we denote 
		\begin{equation}
			\be_\emptyset := C_d \int_{B(0,R_1)} \phi^+ \d y - \int_{\R^d} \phi^-\d y,
		\end{equation}
		then, by \eqref{cond: diluted_pairwise interaction}, we have $\be_\emptyset > 0$. Consequently,
		\begin{align*}
			E_0(\omega) \geq  \int_{L_R(\omega) \cap T_0} \be_\emptyset \d x= |L_R(\omega) \cap T_0| \be_\emptyset.
		\end{align*}
		Therefore, we have $E_0 \geq \overline{E}_0$. Let us consider a contour $\gamma$ and a configuration $\omega$ that achieves this contour, for any domino $(i,j) \in D(\gamma)$ we have that $T_j \subset L_R(\omega)$ and thus
		\begin{equation}
			E_j(\omega) \geq \delta^d \be_\emptyset >0. 
		\end{equation}
		Therefore, we can conclude using the dominoes approach that the interaction satisfies a Peierls-like condition. 
		\newline
		
		Now we consider the case where $R < R_1$, we will adopt a global approach to the analysis of the energy inside each contour. Let us consider $\epsilon > R$ which we will fix close enough to $R$ and $\delta \leq \frac{R-\epsilon}{2\sqrt{d}}$. Before the proof of the Peierls condition, we need the following geometrical lemma. 
		
		\begin{lemma}\label{lem: geometrique_theta_epsilon}
			For $0<\epsilon < R \leq R_1$, we define $\theta_\epsilon$ as
			\begin{align*}
				\theta_\epsilon := \inf\limits_{\substack{\omega \in \Omega_f \\ \gamma : V_{\omega,\gamma, \epsilon}>0}}\left\{ \frac{V_{\omega,\gamma, \epsilon}}{V_{\omega,\gamma, R_1} - V_{\omega,\gamma, \epsilon}} \right\}
			\end{align*}
			where
			\begin{align*}
				V_{\omega,\gamma, r} =|\partial L_R(\omega) \oplus B(0,r) \cap L_R(\omega) \cap \widehat{\gamma} |. 
			\end{align*}
			Then we have
			\begin{align*}
				\theta_\epsilon = \frac{\epsilon^d}{R_1^d - \epsilon^d}.
			\end{align*} 
		\end{lemma}
		\begin{proof}
			For any contour $\gamma$ and any configuration $\omega$ that achieves this contour, we observe that 
			\begin{align*}
				\partial L_R(\omega) \oplus B(0,r) \cap L_R(\omega) \cap \widehat{\gamma} = L_R(\omega)^c \oplus B(0,r) \cap \widehat{\gamma} \quad \text{if } r \leq R, \\
				\partial L_R(\omega) \oplus B(0,r) \cap L_R(\omega) \cap \widehat{\gamma} \subset L_R(\omega)^c \oplus B(0,r) \cap \widehat{\gamma} \quad \text{if } r > R.
			\end{align*}
			We can approximate $L_R(\omega)^c$ using a union of open balls, so there is $\left((x_i, r_i)\right)_{i\in\N}$ and $K_n = \bigcup_{i=1}^n \mathring{B}(x_i, r_i)$ such that $K_n \subset L_R(\omega)^c$ and $ L_R(\omega)^c \backslash K_n$ is decreasing and converges to $\emptyset$. For a fixed $n$, we have
			\begin{align*}
				\left|K_n \oplus B(0,R_1) \cap \widehat{\gamma}\right| &= \int\limits_{0}^{R_0} H_{d-1}\left(\partial(K_n \oplus B(0,r)) \cap \widehat{\gamma}\right) \d r \\
				&=  S_d \int\limits_{0}^{R_0} \sum_{i=1}^n \alpha_{i,\gamma}(r)  (r_i +r )^{d-1} \d r
			\end{align*}
			where $H_{d-1}$ is the $(d-1)$-Hausdorff measure, $S_d$ is the surface of the unit ball and $\alpha_{i,\gamma}(r)$ is the proportion of the surface of $\mathring{B}(x_i, r_i+r)$ that appears in $\partial(K_n \oplus B(0,r)) \cap \widehat{\gamma}$, 
			\begin{align*}
				\alpha_{i,\gamma}(r)= \frac{H_{d-1} \left(\partial \mathring{B}(x_i, r_i+r) \cap \partial(K_n \oplus B(0,r)) \cap \widehat{\gamma}\right)}{H_{d-1}\left(\partial \mathring{B}(x_i, r_i+r)\right)}.
			\end{align*} 
			Furthermore, $\alpha_{i,\gamma}$ is decreasing. Indeed, let $r'>r$ for $z \in \partial \mathring{B}(x_i, r_i+r')$ that appears in $ \partial (K_n \oplus B(0,r')) \cap \widehat{\gamma}$ then $y = \frac{r}{r'}(z-x_i) + x_i \in \partial \mathring{B}(x_i, r_i+r)$ would appear in  $\partial (K_n \oplus B(0,r)) \cap \widehat{\gamma}$. Even though $\widehat{\gamma}$ is not convex but it is indeed true because the contours are thick enough so that $d_2\left(\partial(K_n \oplus B(0,R)) \cap \widehat{\gamma}, \partial \gamma\right) > 2R$. As such, we have
			\begin{align*}
				\left|K_n \oplus B(0,R_1) \cap \widehat{\gamma}\right| &= S_d \left(\frac{R_1}{\epsilon}\right)^d  \int\limits_{0}^{\epsilon} \sum_{i=1}^n \alpha_{i,\gamma}\left(\frac{R_1}{\epsilon} r\right)  \left(\frac{\epsilon}{R_1} r_i +r \right)^{d-1} \d r \\
				&\leq  S_d \left(\frac{R_1}{\epsilon}\right)^d \int\limits_{0}^{\epsilon} \sum_{i=1}^n \alpha_{i,\gamma}\left( r\right)  \left( r_i +r \right)^{d-1} \d r \\
				&\leq \left(\frac{R_1}{\epsilon}\right)^d \left|K_n \oplus B(0,\epsilon) \cap \widehat{\gamma}\right|.
			\end{align*}
			Since Lebesgue measure is continuous and, for $r \in \{\epsilon, R_1\}$, $K_n \oplus B(0,r) \cap \widehat{\gamma}$ converges to $L_R(\omega)^c \oplus B(0,r) \cap \widehat{\gamma}$ with the Hausdorff metric, we can take the limit as $n$ tends to infinity on both sides of the inequality and we get
			\begin{align*}
				V_{\omega, \gamma, R_1} &\leq \mathcal{V}\left(L_R(\omega)^c \oplus B(0,R_1) \cap \widehat{\gamma}\right) \\
				&\leq \left(\frac{R_1}{\epsilon}\right)^d \left|L_R(\omega)^c \oplus B(0,\epsilon) \cap \widehat{\gamma}\right| \\
				& \leq \left(\frac{R_1}{\epsilon}\right)^d V_{\omega, \gamma, \epsilon}.
			\end{align*} 
			Using the previous inequality, we have
			\begin{align*}
				\frac{V_{\omega,\gamma, \epsilon}}{V_{\omega,\gamma, R_1} - V_{\omega, \gamma, \epsilon}} \geq \frac{\epsilon^d}{R_1^d - \epsilon^d}.
			\end{align*}
			We have equality for $L_R(\omega) =\R^d \backslash \{0\}$ but $\omega$ is not a valid configuration since we would have infinitely many points near the origin. But it can be obtained as a limit of configurations and therefore,
			\begin{align*}
				\theta_\epsilon = \frac{\epsilon^d}{R_1^d - \epsilon^d}.
			\end{align*}
		\end{proof}
		For any $r>0$ and any configuration $\omega \in \Omega$, we denote by $L_R^{-r}(\omega)	:= L_R(\omega) \backslash \partial L_R(\omega) \oplus B(0,r)$. We consider $L\geq R_1 > R > \epsilon$. By construction, we have $L_R(\omega) \backslash L_R^{-R_1}(\omega) \subset \bigcup_{\gamma \in \Gamma(\omega)} \widehat{\gamma \backslash \partial^- \gamma}$. For $x \in L_R^{-R_1}(\omega)$, we know by construction that $B(x,R_1) \subset L_R(\omega)$. Therefore, we completely recover the positive part of $\phi$ and we have
		\begin{align}\label{ineq :prop_peierls_global1}
			\int_{L_R(\omega)-x} \phi \d y \geq \int_{\R^d} \phi \d y = C_\phi.
		\end{align}
		Now, for $x \in L_R^{-\epsilon}(\omega) \backslash L_R^{-R_1}(\omega)$, by construction, we know that $B(x,\epsilon) \subset L_R(\omega)$ and thus obtain the following inequality
		\begin{align}
			\int_{L_R(\omega)-x} \phi \d y &\geq \int_{B(0,\epsilon)} \phi^+ \d y - \int_{\R^d} \phi^- \d y  \nonumber\\
			& \geq C_\phi - \int_{B(0,R_1)\backslash B(0,\epsilon)} \phi^+ \d y. \label{ineq :prop_peierls_global2}
		\end{align}
		Finally, for $x \in L_R(\omega) \backslash L_R^{-\epsilon}(\omega)$, we know that there is $z \in \omega$, the closest to $x$ and that $B_{sec}(z,x) \subset L_R(\omega)$. Since $\phi$ is radial we have 
		\begin{align}
			\int_{L_R(\omega)-x} \phi \d y &\geq \int_{B_{sec}(0,x-z)} \phi^+ \d y - \int_{\R^d} \phi^-\d y \nonumber \\
			&\geq C_d \int_{B(0,R)} \phi^+ \d y - \int_{\R^d} \phi^- \d y. \label{ineq :prop_peierls_global3}
		\end{align}
		By combining inequalities \eqref{ineq :prop_peierls_global1}, \eqref{ineq :prop_peierls_global2}, \eqref{ineq :prop_peierls_global3}, we obtain 
		\begin{align*}
			\begin{split}
				E_{\overline{\gamma}}(\omega) \geq	C_\phi |L_R^{-\epsilon}(\omega) \cap \widehat{\gamma}| + \left(C_d \int\limits_{B(0,R)} \phi^+ \d y - \int\limits_{\R^d} \phi^- \d y \right)  V_{\omega,\gamma, \epsilon} & \\
				- \left(V_{\omega,\gamma, R_1} - V_{\omega, \gamma, \epsilon}\right) & \int\limits_{B(0,R_1) \backslash B(0,\epsilon)} \phi^+ \d y,
			\end{split}
		\end{align*}
		By definition of $\theta_\epsilon$, we get
		\begin{align*}
			\begin{split}
				E_{\overline{\gamma}}(\omega) \geq C_\phi |L_R^{-\epsilon}(\omega) \cap \widehat{\gamma}| + \theta_\epsilon \left(C_d \int\limits_{B(0,R)} \phi^+ \d y- \int\limits_{\R^d} \phi^- \d y -  \frac{1}{\theta_\epsilon} \int\limits_{B(0,R_1) \backslash B(0,\epsilon)} \phi^+ \d y\right) \\
				\times \left(V_{\omega,\gamma, R_1}  - V_{\omega, \gamma, \epsilon}\right).
			\end{split} 
		\end{align*}
		 According to Lemma \ref{lem: geometrique_theta_epsilon}, we have $\frac{1}{\theta_\epsilon} \to (\frac{R_1}{R})^d - 1 $ as $\epsilon$ goes to $R$. Therefore, under assumption \eqref{cond: diluted_pairwise interaction} and for $\epsilon$ close enough to $R$, we have
		\begin{align*}
			C_d \int\limits_{B(0,R)} \phi^+ \d y- \int\limits_{\R^d} \phi^- \d y - \frac{1}{\theta_\epsilon} \int\limits_{B(0,R_1) \backslash B(0,\epsilon)} \phi^+ \d y \geq 0.
		\end{align*}
		As a consequence, we have
		\begin{align*}
			E_{\overline{\gamma}}(\omega) \geq |L_R^{-\epsilon}(\omega) \cap \widehat{\gamma}| C_\phi,
		\end{align*}
		and thus
		\begin{align*}
			E_{\overline{\gamma}}(\omega) - \overline{E}_{\overline{\gamma}}(\omega) \geq C_\phi \left( |L_R^{-\epsilon}(\omega) \cap \widehat{\gamma}| - \delta^d |\gamma_1| \right).
		\end{align*}
		The difference between the volumes is bounded from below by the volume of empty tiles covered by the halo of radius $R-\epsilon$. Since $ \delta \leq \frac{R-\epsilon}{2\sqrt{d}}$, we know that for dominoes $(i,j)$, the presence of a point in $T_i$ assures that $T_j \subset L_{R-\epsilon}(\omega)$ even though $T_j$ is void of point. Therefore, using \cite[Lemma 5]{renaudchanQuermass}, we know there is $r_0>0$ and $v_0 = r_0 \delta^d$ such that
		\begin{align*}
			E_{\overline{\gamma}\backslash \partial^- \overline{\gamma}}(\omega) - \overline{E}_{\overline{\gamma}\backslash \partial^- \overline{\gamma}}(\omega) \geq C_\phi v_0 |\overline{\gamma}|.
		\end{align*}
		
		The existence of a critical activity $z_\beta^c \in O_\beta$ for which the diluted pairwise interaction exhibit a phase transition is therefore a consequence of Theorem \ref{thm: liquid_gaz1}. However, we can be more precise about the localisation of such critical value. Let us consider $a(\beta) = \min\{2, e^{-\beta c}\}$, where $ 0<c < \min\{ \frac{\beta \be_+}{6}, \beta \delta^d C_\phi \}$, and 
		\begin{gather}
			\hat{z}_\beta^- := \frac{\ln(1 + e^{\beta \delta^d C_\phi - a(\beta)})}{\delta^d}, \quad \hat{z}_\beta^+ = \frac{\ln(1 + e^{\beta \delta^d C_\phi + a(\beta)})}{\delta^d}. 
		\end{gather}
		By direct computation, for $\beta$ large enough we have 
		\begin{gather}
			G(\hat{z}_\beta^-) \leq -\frac{a(\beta)}{\delta^d} + 2\eta(\tau, l_0) = -\frac{a(\beta)}{\delta^d} + 4 e^{-\frac{\beta \be_+ l_0}{6} +\frac{8 l_0}{3}} < 0 \\
			G(\hat{z}_\beta^+) \geq \frac{a(\beta)}{\delta^d} - 2\eta(\tau, l_0) = \frac{a(\beta)}{\delta^d} - 4 e^{-\frac{\beta \be_+ l_0}{6} +\frac{8 l_0}{3}} > 0
		\end{gather}
		and thus, we can conclude that $z_\beta^c \in (\hat{z}_\beta^-, \hat{z}_\beta^+)$. Finally, we have  
		\begin{align*}
			\hat{z}_\beta^- - C_\phi \beta = -\frac{a(\beta)}{\delta^d} + \frac{1}{\delta^d} \ln(1+ e^{-\beta C_\phi \delta^d + a(\beta)}) = -\frac{a(\beta)}{\delta^d} + o(a(\beta)) \\
			\hat{z}_\beta^+ - C_\phi \beta = \frac{a(\beta)}{\delta^d} + \frac{1}{\delta^d} \ln(1+ e^{-\beta C_\phi \delta^d - a(\beta)}) = \frac{a(\beta)}{\delta^d} + o(a(\beta)).
		\end{align*}
		As a consequence, we have $|z_\beta^c - C_\phi \beta| = O(e^{-c\beta})$. 
	\end{proof}
	
	\section*{Acknowledgement}
	DD and CRC acknowledge support from the CDP C2EMPI and are grateful to the French State (France-2030 programme), the University of Lille,  the Initiative d'excellence de l'Université de Lille, and the Métropole Européenne de Lille for  funding the R-CDP-24-004-C2EMPI project. CRC also acknowledges support  from Persyval-lab (ANR-11-61 LABX-0025-01).
	
	\bibliographystyle{abbrvnat}
	\bibliography{LGPT}
	
\end{document}